\documentclass[oneside,english]{amsart}
\usepackage{mathptmx}
\usepackage[T1]{fontenc}
\usepackage[utf8]{inputenc}
\usepackage{fancyhdr}
\usepackage{babel}
\usepackage{url}
\usepackage{amsthm}
\usepackage{amssymb}
\usepackage[unicode=true,pdfusetitle,
 bookmarks=false,
 breaklinks=false,pdfborder={0 0 1},backref=false,colorlinks=false]
 {hyperref}
\usepackage[capitalise]{cleveref}

\numberwithin{equation}{section}

\theoremstyle{plain}
\newtheorem{thm}{Theorem}
\newtheorem{prop}[thm]{Proposition}
\newtheorem{lem}[thm]{Lemma}

\newtheorem{cor}[thm]{Corollary}
\theoremstyle{definition}
\newtheorem{defn}[thm]{Definition}

\newtheorem{exa}[thm]{Example}
\theoremstyle{remark}
\newtheorem{rem}[thm]{Remark}
\newtheorem*{acknowledgement*}{Acknowledgemnents}

\usepackage{mathrsfs}
\usepackage{amssymb}
\usepackage{dsfont}
\usepackage{verbatim}
\usepackage{url}
\usepackage{mathtools}
\usepackage{etoolbox}
\usepackage{leftidx}

\def\Z{\mathbb {Z}}
\def\Q{\mathbb {Q}}
\def\R{\mathbb {R}}
\def\C{\mathbb {C}}

\newcommand\Qp{\Q_p}
\newcommand\Zp{\Z_p}

\def\embed{\hookrightarrow}

\def\isom{\simeq}

\def\bs{\backslash}

\DeclareMathOperator{\diag}{diag}

\DeclareMathOperator{\Gal}{Gal}

\DeclareMathOperator{\id}{id}
\DeclareMathOperator{\Id}{Id}

\DeclareMathOperator{\Lie}{Lie}

\DeclareMathOperator{\Res}{Res}

\DeclareMathOperator{\Span}{Span}

\DeclareMathOperator{\supp}{supp}

\DeclareMathOperator{\Tr}{Tr}

\DeclareMathOperator{\vol}{vol}

\newcommand\abs[1]{\left| {#1} \right|}
\newcommand\norm[1]{\left\Vert {#1} \right\Vert}

\newcommand{\xdashrightarrow}[2][]{\ext@arrow 0359\rightarrowfill@@{#1}{#2}}

\DeclareRobustCommand
  \rddots{\mathinner{\mkern1mu\raise\p@
    \vbox{\kern7\p@\hbox{.}}\mkern2mu
    \raise4\p@\hbox{.}\mkern2mu\raise7\p@\hbox{.}\mkern1mu}}

\def\Ps{\mathcal{P}}
\newcommand{\Cc}{C_{\mathrm{c}}}

\newcommand{\Cic}{\Cc^{\infty}}

\newcommand\SL{\mathrm{SL}}
\newcommand\GL{\mathrm{GL}}

\newcommand\SO{\mathrm{SO}}

\newcommand\SU{\mathrm{SU}}

\newcommand\lieg{\mathfrak{g}}

\newcommand\A {\mathbb{A}} 
\newcommand\BB{\mathbb{B}}

\newcommand\DD{\mathbb{D}}

\newcommand\GG{\mathbb{G}}
\newcommand\HH{\mathbb{H}}

\newcommand\bLL{\mathbb{L}}
\newcommand\MM{\mathbb{M}}

\newcommand\PP{\mathbb{P}}
\newcommand\QQ{\mathbb{Q}}
\newcommand\RR{\mathbb{R}}
\newcommand\bSS{\mathbb{S}}
\newcommand\TT{\mathbb{T}}

\newcommand\XX{\mathbb{X}}

\newcommand\ZZ{\mathbb{Z}}

\newcommand\bH{\mathbf{H}}

\DeclareMathAlphabet{\mathcal}{OMS}{cmsy}{m}{n}

\newcommand\calH{\mathcal{H}}

\newcommand\calO{\mathcal{O}}
\newcommand\calP{\mathcal{P}}

\newcommand\calV{\mathcal{V}}

\newcommand\diff{\mathop{}\!\mathrm{d}}

\newcommand\dvol{\diff\!\vol}

\newcommand\dcross{\diff^{\kern-2pt\raisebox{1pt}{$\times$}}\kern-8pt}

\newcommand\Adele{\mathbb{A}}

\newcommand\AF{\Adele_F}

\newcommand\Af{\Adele_\mathrm{f}}

\newcommand\GF{\GG(F)}
\newcommand\Gi{G_\infty}
\newcommand\Gp{G_p}
\newcommand\Gv{G_v}

\newcommand\GA{\GG(\Adele)}
\newcommand\GAF{\GG(\AF)}
\newcommand\GAf{\GG(\Af)}

\newcommand\Kp{K_p}

\newcommand\Kf{K_\mathrm{f}}

\newcommand\OF{\calO_F}
\newcommand\Ov{\calO_v}

\newcommand\gf{g_\mathrm{f}}
\newcommand\gi{g_\infty}
\newcommand\kf{k_\mathrm{f}}

\newcommand\Qbar{\bar{\Q}}

\def\ae{a.e.\ }

\def\ie{i.e.\ }

\def\resp{resp.\ }

\newcommand\blfootnote{\xdef\@thefnmark{}\@footnotetext}

\DeclareFontFamily{U}{wncy}{}
\DeclareFontShape{U}{wncy}{m}{n}{<->wncyr10}{}
\DeclareSymbolFont{mcy}{U}{wncy}{m}{n}
\DeclareMathSymbol{\Sha}{\mathord}{mcy}{"58} 

\newcommand\mb{\bar\mu}

\newcommand{\tD}{\tilde{D}}
\newcommand{\tGG}{\tilde{\GG}}
\newcommand{\tHH}{\tilde{\HH}}
\newcommand{\Ggt}{G_{\infty,\geq2}}
\newcommand{\tGgt}{\tGG_{\geq2}}
\newcommand{\Gpo}{G_{p,1}}
\newcommand{\Kpo}{K_{p,1}}
\newcommand{\Gpgt}{G_{p,\geq2}}

\newcommand{\Ue}{{U_\epsilon}}
\newcommand{\Vd}{{V_\delta}}
\DeclareMathOperator{\denom}{denom}

\begin{document}

\title{Arithmetic quantum unique ergodicity for products of hyperbolic $2$- and $3$-spaces}

\author{Zvi Shem-Tov}
\address{Einstein Institute of Mathematics, The Hebrew University of Jerusalem,  Israel}
\email{\url{zvi.shem-tov@mail.huji.ac.il}}

\author{Lior Silberman}
\address{Department of Mathematics, University of British Columbia, Vancouver\ \ BC\ ~V6T 1Z2, Canada}
\email{\url{lior@math.ubc.ca}}
\thanks{LS was supported by an NSERC Discovery Grant}

\begin{abstract}
We prove the arithmetic quantum unique ergodicity (AQUE) conjecture for
sequences of Hecke--Maass forms on quotients
$\Gamma\bs (\HH^{(2)})^r \times (\HH^{(3)})^s$.  An argument by induction
on dimension of the orbit allows us to rule out the limit measure
giving positive mass to closed orbits of proper subgroups despite many returns of
the Hecke correspondence to neighborhoods of the orbit.
\end{abstract}


\maketitle
\tableofcontents{}


\section{Introduction}
\subsection{Statement of result}
We prove the Arithmetic Quantum Unique Ergodicity Conjecture for hyperbolic
$3$-manifolds.  More generally fix
integers $r,s\geq 0$ so that $r+s\geq 1$ and consider the symmetric space 
$$S = (\HH^{(2)})^r \times (\HH^{(3)})^s,$$
where $\HH^{(n)}$ denotes hyperbolic $n$-space.  Let $\Gamma$ be a lattice
in the isometry group 
\begin{equation}\label{eq:group}
G = \SL_2(\R)^r \times \SL_2(\C)^s.
\end{equation}  
Then the finite-volume manifold $Y = \Gamma \bs S$ is equipped with a family of
$r+s$ commuting differential operators coming from the Laplace--Beltrami
operator in each factor (the sum of which is the Laplace--Beltrami
operator of $Y$). Assume that $\Gamma$ is, in addition, a congruence
lattice (e.g. $\SL_2(\Z[i])\subset \SL_2(\C)$) so that the manifold $Y$ is also equipped with a ring of discrete averaging
operators, the Hecke operators, which further commute with the
differential operators noted above (see \cref{sec:adelic-quotient} for a
discussion of congruence lattices and \cref{sec:Hecke} for the construction
of the Hecke algebra).
Write $\dvol$ for the Riemannian volume element normalized to
have total volume $1$.  
\begin{thm}\label{thm:main1}
Let $\psi_j\in L^2(Y)$ be a sequence of normalized joint eigenfunctions of
both the ring of invariant differential operators and of the Hecke operators.
Assume that the Laplace-eigenvalues $\lambda_j\to \infty$.
Then the probability measures $\abs{\psi_j}^2(y)\dvol(y)$ converge in the weak* topology to $\dvol$. 
\end{thm}
(Recall that $\mu_i$ converge weak* to $\mu$ if for every continuous function
$f$ with compact support,
$$
\int fd\mu_i\to\int fd\mu
$$
as $i\to\infty$.)

The case $r=1$, $s=0$ (i.e. hyperbolic surfaces) is due to Lindenstrauss
\cite{Lindenstrauss:SL2_QUE}.  When $s=0$ and $r>1$ (more generally, when
the eigenvalues of $\psi_j$ with respect to the Laplace operator corresponding
to one of the $\HH^{(2)}$ factors tend to infinity)
the claim follows from existing
techniques\footnote{We adapt previous positive entropy arguments to groups
over number fields in \cref{sec:homogeneous} and the claim then follows from
the measure classification result cited there since $\SL_2(\R)$ does not have
proper semisimple subgroups.  In some cases where the eigenvalue tends
to infinity in one factor but not in another the result was already
established even in the non-arithmetic setting by \cite{BrooksLindenstrauss:OneHecke}.}.
When $s>0$ those existing techniques place significant constraints on the
possible ergodic components of the limit measures under consideration but fall
short of proving that the limit measure is uniform, and we introduce
a new method for eliminating components other than the uniform measure.  This
is the main innovation of our paper; to state its realization here let $G_i$
be one of the factors isomorphic to $\SL_2(\C)$ in the product \eqref{eq:group}
defining $G$.
Identifying $G_i$ with $\SL_2(\C)$ for the moment let
$H_i = \SL_2(\R)$, and let $M_i = \{ \diag(e^{i\theta},e^{-i\theta})\}$
be the group of diagonal matrices with entries of modulus $1$.
Finally set $H = H_i \times \prod_{j\neq i} G_j$ (i.e. multiply $H_i$  
at the $i$th factor with the full isometry groups $G_j$ at the other factors). 
\begin{thm}\label{thm:main2}
Let $\mu$ be a measure on $X=\Gamma\bs G$ which is the weak-* limit of measures
$$\mu_j = \abs{\phi_j}^2 \dvol_X$$
where the $\phi_j \in L^2(X)$ are normalized Hecke eigenfunctions.
If $g=(g_j)_j \in G$ is such that the entries of $g_i$ are algebraic numbers
then
$$\mu(\Gamma gHM_i) = 0\,.$$
\end{thm}

\subsection{Context: QUE Conjectures for locally symmetric spaces}
More generally, let $G$ be a semisimple Lie group, $K\subset G$ a maximal
compact subgroup, $\Gamma<G$ a lattice, and consider the locally symmetric
space 
$$
Y=\Gamma\bs G/K,
$$ 
with its uniform probability measure $dy$. The ring of $G$-invariant
differential operators on $G/K$ is commutative and commutes with the action
of $\Gamma$, hence acts on functions on $Y$.  Let
$\{\psi_j\}_{j\geq 1} \subset L^2(Y)$ be an orthonormal sequence of joint
eigenfunctions of this ring.  We remark that the Laplace--Beltrami operator
of $Y$ belongs to the ring and that writing $\lambda_j$ for the corresponding
eigenvalue for $\psi_j$ we necessarily have $\abs{\lambda_j}\to\infty$.
The \emph{Quantum Unique Ergodicity} (QUE) Conjecture for locally symmetric
spaces asserts the sequence $\{\psi_j\}$ becomes equidistributed on $Y$,
in the sense that
$$
\int_Yf(y)|\psi_j(y)|^2dy\to \int_Yf(y)dy,
$$
for all test functions $f\in C_c(Y)$.  In other words, the sequence
$|\psi_j|^2dy$ of probability measures converges in the weak-$*$ topology to
the uniform measure $dy$.  Next, when $\Gamma$ is a \emph{congruence}
lattice the space $Y$ also affords a commutative family of discrete averaging
operators, the \emph{Hecke operators}, which commute with the invariant
differential operators.  The \emph{Arithmetic} QUE Conjecture (AQUE)
is the restricted form of the QUE Conjecture for sequences of Hecke
eigenfunctions, that is for sequences where each $f_j$ is simultaneously
an eigenfunction of the ring of invariant differential operators and of the
ring of Hecke operators.

Investigation of Arithmetic QUE goes back to the work
\cite{RudnickSarnak:Conj_QUE} of Rudnick--Sarnak, who formulated
the QUE conjecture for hyperbolic surfaces and, more generally, for compact
manifolds of negative sectional curvature.  A major breakthrough on
this problem was due to Lindenstrauss, who in \cite{Lindenstrauss:SL2_QUE}
established AQUE for congruence hyperbolic surfaces, that is congruence
quotients $Y=\Gamma\bs \HH^{(2)}$ of the hyperbolic plane.
More precisely, Lindenstrauss proved that weak-$*$ limits as above were
proportional to the uniform measure.  This fully resolved the Conjecture
in the case of \emph{uniform} lattices, \ie when the quotient is compact.
For non-uniform lattices, however, the possibility remained that the
constant of proportionality was strictly less than one ("escape of mass"),
an alternative ruled out later by Soundararajan \cite{Sound:EscapeOfMass}.
Silberman--Venkatesh
\cite{SilbermanVenkatesh:SQUE_Lift,SilbermanVenkatesh:AQUE_Ent}
obtains generalizations of some of this work to the case of general
semisimple groups $G$, formulating the QUE and AQUE conjectures in the
context of locally symmetric spaces and obtaining further cases of AQUE.

\subsection{Discussion}
As stated above, we establish AQUE for congruence quotients of
products of hyperbolic 2- and 3-spaces.  The case
$Y=\SL_2(\ZZ[i])\bs \HH^{(3)}$ is already new and contains most of the 
the novel ideas of this paper; reading the paper with this assumption
in mind will give the reader most of the insight but avoid the technical
legerdemain needed to handle number fields with multiple infinite places.
For the rest of the introduction we concentrate on this case.

Thus let $G=\SL_2(\C)$ acts transitively by isometries on 
the space $S=\HH^{(3)}$ with point stabilizer the maximal compact subgroup
$K = \SU(2)$.  Let $\Gamma = \SL_2(\ZZ[i])$, the group of unimodular matrices
with Gaussian integer entries, which is indeed a lattice in $G$.
Then $Y=\Gamma\bs S$ is a finite-volume hyperbolic $3$-manifold (or more precisely, an orbifold) and 
$X=\Gamma\bs G$ is its frame bundle.  The action of
$A = \{ \diag(a,a^{-1}) \mid a>0 \}$ on $X$ from the right corresponds
to the geodesic flow on the frame bundle, with the commuting subgroup
$M = \{ \diag(e^{i\theta},e^{-i\theta}) \}$ acting by rotating the frame
around the tangent vector.

Let $H=\SL_2(\R) \subset G$.  Then $Y_H = \SL_2(\Z)\bs H / \SO(2) \subset Y$
is a finite-volume hyperbolic surface embedded in $Y$; its unit tangent bundle
is the finite-volume $H$-orbit $X_H = \SL_2(\Z)\bs H \subset X$.  Very relevant
to us will be the subset $\SL_2(\Z)\bs H M$, corresponding to the pullback
of the frame bundle of the hyperbolic $3$-fold to the surface.
Geometrically the unit tangent bundle of $Y_H$ embeds in the frame bundle of
$Y$ in multiple ways invariant by the geodesic flow, parametrized by choosing
a normal vector to the tangent vector at one point.  The set of choices
is thus parametrized by the group $M$ rotating the frame at the chosen point
around the tangent vector.

The Casimir element in the universal enveloping algebra of the Lie algebra $G$
acts on functions on $X$ and on $Y$, where it is proportional to the
Laplace--Beltrami operator.  In addition there is a family of
discrete averaging operators acting on functions on $X$ and commuting with the
right $G$ action, hence also on functions on $Y$
(thought of as $K$-invariant functions on $X$).  These \emph{Hecke operators}
can be constructed as follows: for each $g\in \SL_2(\Q(i))$
the set $[g] = \Gamma \bs \Gamma g \Gamma \subset \Gamma \bs \SL_2(\Q(i))$
is finite, so for a function $f$ on $X$ we can set
\begin{equation}
(T_gf)(x)=\sum_{s\in[g]}f(sx)\,.
\end{equation}
Since these operators are $G$-equivariant they also commute with the Casimir
element, and hence with the Laplace--Beltrami operator on $Y$.
They are bounded (the $L^2$ operator norm of $T_g$ is at at most the cardinality
of $[g]$) and it is not hard to check that the adjoint of $T_g$ is
$T_{g^{-1}}$.  It is a non-trivial fact that the $T_g$ commute with
each other, and hence are a commuting family of bounded normal operators.
For the sequel we will take a different point of view that gives better
control of these operators; see the construction in \cref{sec:Hecke}.

The proof of \cref{thm:main1} is given in \cref{sec:homogeneous}.
Our strategy is the one laid down by Lindenstrauss and followed by later work
on AQUE.
\begin{enumerate}
\item\label{microlocalstep}
The microlocal lift of \cite{SilbermanVenkatesh:SQUE_Lift}
(see also \cite{Lindenstrauss:HH_QUE} for the case $s=0$; the original
such constructions in much greater generality are due to
Shnirelman, Zelditch and Colin de Verdière
\cite{Schnirelman:Avg_QUE,Zelditch:PseudoDiffCalHypSurfaces,CdV:Avg_QUE})
shows that any weak-* limit on $Y$ as in \cref{thm:main1} is the projection
from $X$ of a limit $\mu$ as in \cref{thm:main2} which, in addition, is an
$AM$-invariant measure.
\item\label{entropystep} We show that almost every $A$-ergodic component
of $\mu$ has positive entropy under the $A$-action.
These arguments are standard (see
\cite{SilbermanVenkatesh:AQUE_Ent} generalizing
\cite{LindenstraussBourgain:SL2_Ent}) but we need to adjust them to handle
the fact that $M$ isn't finite and that the group is essentially defined
over a number field.  What is shown is that the measure of an
$\epsilon$-neighbourhood of a (compact piece) of
an orbit $xAM$ in $X$ decays at least as fast as $\epsilon^h$ for some $h>0$.
\item The classification of $A$-invariant measures in
\cite{EinsiedlerLindenstrauss:GeneralLowEntropy} implies that all the ergodic
components of $\mu$ other than the $G$-invariant measure
are contained in sets of the form $xHM$ where $xH$ is a finite-volume $H$-orbit
in $X$.
\item\label{newstep} In a closed $H$-orbit $\Gamma gH\subset X$ the entries of
$g$ are algebraic numbers, so by \cref{thm:main2} the exceptional possibilities
are contained in a countable family of sets each of which has measure zero.
It follows that $\mu$ is $G$-invariant.
\item It was shown in Zaman's M.Sc.\ Thesis \cite{Zaman:EscapeThesis} that
the measure $\mu$ is a probability measure (a generalization of Soundararajan's
proof \cite{Sound:EscapeOfMass} for the case of $\SL_2(\Z)\bs\SL_2(\R)$).
\end{enumerate}

The bulk of the paper is then devoted to realizing step \ref{newstep}.  Before
discussing how that is achieved, let us remark on its necessity.  
The group $G=\SL_2(\R)$ has no proper semisimple subgroups, so the issue
of ruling out components supported on orbits of such subgroups did not arise
in the original work of Lindenstrauss.  In followup
generalizations to higher-rank groups the difficulty was addressed by
choosing the lattice $\Gamma$ 
so that the subgroups $H$ that could arise did not have finite-volume orbits
on $X$ and $G$ itself was chosen $\R$-split so that $M$ was finite and could
be ignored.

We now go further.  The group $G=\SL_2(\C)$ has $H=\SL_2(\R)$ as a relevant
subgroup, forcing us to confront the problem of ruling out components
supported on finite-volume $H$-orbits head-on.  It is also not $\R$-split,
forcing us to contend with the infinitude of $M$.  Indeed,
equip a finite-volume $H$-orbit $xH$ with the $H$-invariant probability
measure $\nu$.  Translating $\nu$ by $m\in M$ gives further
$A$-invariant probability measures supported on the subsets $xHm$, and
averaging these measures to gives an $A$-invariant probability measure
supported on $xHM$.
The construction in step \ref{microlocalstep} is such that if $\nu$ occurs
in the ergodic decomposition of $\mu$ then so do its $M$-translates, so
we need to rule out the averaged measure on $xHM$ as occuring in $\mu$
(fortunately there are
only countable many such subsets).  Now previous technology as mentioned
in step \ref{entropystep} could show that ($\epsilon$-neighborhoods of
pieces of) subgroup orbits such as $xH$ have small
measure
but $xHM$ is not of this form.  Instead the subgroup generated by $HM$
is all of $G$, so the only orbit of a closed subgroup containing $xHM$
is all of $X$.

\subsection{Sketch of the proof of \cref{thm:main2}}
Let $U\subset xHM$ be a bounded open set.  Then showing that $\mu(xHM)=0$
amounts to bounding the mass Hecke eigenfunctions can give to
$\epsilon$-neighbourhoods $U_\epsilon$.  In other words, our goal is
to prove

\begin{equation}\label{eq:toshow}
\int_{\Ue}|\phi(x)|^2dx = o_\epsilon(1),
\end{equation}
for any Hecke-eigenfunction $\phi$, \emph{uniformly in} $\phi$.

We find a Hecke operator (``amplifier'') $\tau$ which acts on
$\phi$ with a large eigenvalue $\Lambda$ yet geometrically ``smears''
$\Ue$ around the space. 
The operator $\tau$ is a linear combination of operators $T_g$ so there is a
subset $\supp(\tau)\subset\SL_2(\Z[i])\bs \SL_2(\Q(i))$ so that
$$
\Lambda \phi(x) = (\tau\star \phi)(x) = \sum_{s\in\supp(\tau)} \tau(s).\phi(sx).
$$
Thus if $\abs{\tau(s)}\le 1$ for each $s\in\supp(\tau)$ then by Cauchy--Schwartz 
$$
\Lambda^2\int_\Ue\abs{\phi(x)}^2dx\le \#\supp(\tau)\int_\Ue\sum_{s\in\supp(\tau)}\abs{\phi(sx)}^2dx,
$$
so that

\begin{align}\label{naivedispersion} \mu_\phi(\Ue)
& \leq \frac{\#\supp(\tau)}{\Lambda^2}
       \sum_{s\in\supp(\tau)} \mu_\phi(s.\Ue) \\
& \leq \frac{\#\supp(\tau)}{\Lambda^2} \max_{s\in\supp(\tau)}
\#\{s'\in\supp(\tau) \mid s.\Ue\cap s'.\Ue\neq\emptyset\}
       \mu_\phi(\bigcup_{s\in\supp(\tau)} s.\Ue) \nonumber\\
& \leq \frac{\#\supp(\tau)}{\Lambda^2} \max_{s\in\supp(\tau)}
\#\{s'\in\supp(\tau) \mid s.\Ue\cap s'.\Ue\neq\emptyset\}\,.\nonumber
\end{align}

The best kind of smearing is thus if there were very few non-empty
intersections between the translates $s.\Ue$, where a bound would
follow as long as we could arrange for $\Lambda^2$ to be large
compared with the size of the support of $\tau$. That is too much
to hope for, but with better geometric and spectral arguments one requires
less stringent control over the intersections.
For various implementations of this strategy see
\cite{RudnickSarnak:Conj_QUE,LindenstraussBourgain:SL2_Ent,SilbermanVenkatesh:AQUE_Ent,BrooksLindenstrauss:OneHecke,ShemTov:OnePlaceHighRank},  
all in settings where $U$ is a piece of an orbit $xL$ for a subgroup $L$.
Then an intersection between $s.\Ue$ and $\Ue$ (say) implies that $s$ is
$O(\epsilon)$-close to the bounded neighbourhood $U U^{-1}$ in
$xLx^{-1}$ and if $\tau$ is such that the elements $s$ in its support have sufficiently
small denominators relative to $\epsilon$ then all the
$s$ causing the intersections
are jointly contained in a single conjugate of $L$.
For certain kinds of groups $L$ (e.g. tori) this implies that there will
be very few intersections.  There are also exceptional cases
("division algebras of prime degree", see
\cite[Prop.\ 4.10]{SilbermanVenkatesh:AQUE_Ent}) where the lattice $\Gamma$
is such that even for larger groups $L$ the set of $s\in\Gamma$ which create 
an intersection jointly generate a torus and again there are few intersections.

In the case we consider, namely $U\subset xHM$, the set $U U^{-1}$ (let alone $U_\epsilon U_\epsilon^{-1}$) is an
\emph{open} subset of
$G$ (a reflection of the fact that $HM$ generates $G$), so the number of
intersections may be large: $s \in U U^{-1}$ holds generically (rational points of large height are equidistributed in $G$),
and the argument fails.
Instead we bound the measure of the pieces through an induction argument on
their dimension. For this lift the picture to $G$, replacing $x\in X$ with a
representative $g\in G$ and $L$ with the submanifold $L = gHM$ which is,
in fact, an irreducible real algebraic subvariety of $G$.
An intersection $sL\cap L$ with $s\in\SL_2(\Q(i))$ is then also
a subvariety, and hence one of two possibilities must hold:
either the intersection has strictly smaller dimension,
in which case we call the intersection (and, by abuse of language, the element
$s$) \emph{transverse}, or it is not, in which case we call
$sL$ (and $s$) \emph{parallel} to $L$ and must actually have $sL = L$ by the
irreducibility of $L$.  By induction we may assume that the measures of all
transverse intersections $s.\Ue\cap \Ue$ are small (one needs to show that
$s.\Ue\cap \Ue$ is contained in a decreasing neighbourhood of $s.U\cap U$).
On the other hand the parallel elements stabilize the subvariety $L$
and this forces them to lie in a proper subgroup $S \subset \SL_2(\Q(i))$
(here we gained over the naive attempt which considered
the much larger subgroup generated by $\SL_2(\Q(i))\cap L L^{-1}$).

For example, the stabilizer of $HM$ under the left action of $G$ on itself
is exactly $H$, so parallel intersections can only arise from $s\in\SL_2(\Q)$.
We now introduce a final idea, allowing us to deal with a situation where
the proper subgroup $S$ noted above is not a torus.  Let $p\equiv 1\,(4)$ be a prime number
so that $p=(a+bi)(a-bi)$ for integers $a,b$ and set
$$g_p = \diag(a+bi,\frac{1}{a+bi})\,.$$
It turns out that for any Hecke eigenfunction $\phi$, the eigenvalue of
$\lambda_p$ of one of the two operators $T_{g_p}$, $T_{g_p^2}$
is at least comparable to the square root of the size of its support.
In addition, the supports $[g_p],[g_p^2]$ of these operators do not
intersect $\SL_2(\Q)$: complex conjugation exchanges the Gaussian prime
$a+bi$ with the distinct Gaussian prime $a-bi$, so the
ratio of matrices from the two double cosets must contain both primes
and cannot be real), so these operators themselves do not cause intersections.  
Making use of these as building blocks and combining the contribution from
many primes we then construct an amplifier which has good spectral properties
and at the same time avoids intersections caused by $\SL_2(\Q)$.

Each closed orbit $xHM$ is the projection of $L=gHM$ where the entries of $g$
are algebraic but need not be rational, so greater care must be taken to
define the complex conjugation which limits the intersection and correspondingly
the arithmetic progression from which we select the primes needs to be
smaller.  Also, we need to consider the stabilizers of the subvarieties
of $L$ that will appear in the recursive argument.
We show each of these stabilizers is either contained
in a conjugate of $H$, and then similarly to the case of $H$ itself we can avoid it entirely, 
or are tori (which are already known to cause few intersections).

\subsection{Organization of the paper}
In \cref{sec:notation} we fix our notation and examine the algebraic structure of forms of $\SL_2$ over number fields, constructing
the complex conjugations which control returns of Hecke translates
to real submanifolds.  We take the adelic point of view of the Hecke operators,
which is more convenient than the one used for the introduction.
In \cref{sec:homogeneous} we then reduce \cref{thm:main1} to \cref{thm:main2}.
Before giving the proof of \cref{thm:main2} in \cref{sec:proofs} we devote
\cref{sec:smallness} to classifying the stabilizers of the subvarieties that
can occur in the recursion and \cref{sec:amp} to constructing the amplifier.

\begin{acknowledgement*}
We would like to thank Elon Lindenstrauss and Manfred Einsiedler for very
useful suggestions and fruitful discussions. 
This work forms part of the PhD
thesis of the first author at the Hebrew University of Jerusalem.
The first author was supported by the ERC grant HomDyn no. 833423.

\end{acknowledgement*}

\section{Notations and background}\label{sec:notation}
For a comprehensive reference on the theory of algebraic groups over number
fields see \cite{PlatonovRapinchuk:AlgGps_NT}.  The theory of automorphic
forms over such groups is developed, for example, in the textbook
\cite{Bump:AutForms},  The particular case of forms of $\SL_2$ over number
fields is articulated in detail in \cite{Gelbart:AutFormsGL2Book}.

\subsection{Unit groups of quaternion algebras}
Varieties (including algebraic groups) will be named in blackboard bold font,

Fix a number field $F$, and let $V = \abs{F}$ be the set of places of $F$,
$V_\infty \subset V$ the set of archimedian places, divided into real
and complex places as $V_\infty = V_\R \sqcup V_\C$.  For a place $v\in V$
write $F_v$ for the completion of $F$ at $v$.  If the place is finite write
$\Ov\subset F_v$ for the maximal compact subring
$\{ x\in F_v \mid \abs{x}_v \leq 1\}$, and let $\varpi_v \in \Ov$
be a uniformizer.  We normalize the associated absolute value so that
$q_v = \abs{\varpi^{-1}}_v$ is the cardinality of the residue field
$\Ov/\varpi \Ov$. For a rational prime $p$ let $F_p = \prod_{v|p} F_v$ and
similarly set $F_\infty = \prod_{v\mid\infty} F_v$.
We write $\Af$ for the ring of finite adeles $\prod'_{v<\infty} (F_v:\Ov)$
and $\AF$ for the ring of adeles $F_\infty \times \Af$.

If $\GG$ is an $F$-variety and $E$ is an $F$-algebra write $\GG(E)$
for the $E$-points of $\GG$, equipped with the analytic topology if $E$ is a
local field extending $F$.  For a place $v$ of $F$ we also write
$G_v = \GG(F_v)$.  We further write $G = G_\infty$ for the manifold 
(Lie group) $\prod_{v\mid\infty} G_v$.

Let $D$ be a quaternion algebra over $F$, that is either the matrix algebra
$M_2(F)$ or a division algebra of dimension $4$ over $F$ (for a first reading
one can consider the special case $F = \Q(i)$, $D = M_2(F)$).  Let $\DD$ be the variety such that
$\DD(E) = D\otimes_F E$, and let $\det\colon\DD\to \A^1$ be the reduced norm.
Let $\GG = \DD^1$ be the algebraic group of elements of reduced norm $1$ in $\DD$:
$$\GG(E) = \{ g\in D\otimes_F E \mid \det(g) = 1 \}\,.$$

For each place $v\in V$ we have the algebra
$D_v = D\otimes_F F_v$.  When $v$ is complex necessarily $D_v \isom M_2(\C)$
and thus $G_v\isom \SL_2(\C)$.  When $v$ is real $D_v$ is either the split
algebra $M_2(\R)$ or Hamilton's quaternions $\bH$, and correspondingly $G_v$ is
one of the groups $\SL_2(\R)$ and $\bH^1 \isom \SU(2)$.  We suppose there are
$s$ complex places and $r+t$ real places divided into $r$ of the first form
and $t$ of the second so that
$$G = G_\infty \isom \SL_2(\R)^r \times \SL_2(\C)^s \times \SU(2)^t$$.

\subsection{Factoring $\Gi$ over a number field; complex conjugation}
We have the factorization $\Gi = \prod_{v\mid\infty} \Gv$. 
For a \emph{complex} place $v\in V_\C$, the usual extension of
scalars realizes $\Gv$ as the group of complex points of the $\C$-group
$\GG\times_F \C$.  However, when thought of as a Lie group this group has
closed subgroups which are not complex -- the key example for us being
the subgroup $\SL_2(\R)$ of $\SL_2(\C)$.  We thus would like to think of the
group $\Gv$ as the group of real points of an algebraic group defined over an
extension of $F$.

Let $N$ be any finite Galois extension of $\Q$ containing $F$. 
Given a complex place $v$ we construct a number field $E\subset N$ of index $2$ equipped with a real place $w$ and an algebraic group $\tGG$ defined over $E$ factoring as an $E$-group
into a product $\tGG = \prod_{i=1}^{r+t+2s} \tGG_i$ such that:

\begin{enumerate}
\item The group $\GF$ embeds in $\tGG(E)$.
\item We have an isomorphism
$$\tGG(E_w) = \prod_{i=1}^{r+t+s} \tGG_i(E_w) \to \Gi$$
where the factors correspond.
\item For a rational prime $p$ splitting completely in $N$ we have an embedding $E\embed \Qp$ such that
for $1\leq i\leq s$ we have $\tGG_i(\Qp) \isom (\SL_2(\Qp))^2$, for
$s<i\leq r+t+s$ we have $\tGG_i(\Qp) \isom \SL_2(\Qp)$ and the resulting
factorization of $\tGG(\Qp)$ is isomorphic to the factorization of $G_p$
into $2s+r+t$ copies of $\SL_2(\Qp)$ coming from the $2s+r+t$ places over $F$
over $\Qp$. We fix an embedding of $E$ in $\Q_p$ as above and use it to view $E$ as a subfield of $\Q_p$. 
\item The latter two identifications are compatible with the first embedding, in
such a way that if an element $\gamma \in \GF$ embeds to a real-valued matrix
in $\tGG_1(E_v)$ (say) then its image in $\tGG_1(\Qp)$ lies in the diagonal
subgroup of $\tGG_1(\Qp) \isom (\SL_2(\Qp))^2$.  This will allow us to choose
Hecke operators which "avoid" a real subgroup in a particular complex place.
\end{enumerate}

Thus let $v$ be a complex place of $F$ and let $N$ be any finite Galois 
extension of $\Q$ containing $F$
(in \cref{sec:homogeneous} $N$ will the Galois closure of a splitting field
of $D$; in \cref{sec:amp} we make a different choice depending on the
subvariety we are trying to avoid).
Let $w$ be a place of $N$ extending $v$, let $c_w\in\Gal(N/\Q)$
the element acting as complex conjugation in the completion of $N$ at $v$,
let $E = N^{c_w}$ be the fixed field of $c_w$, and also write $w$ for the
restriction of $w$ to $E$ -- a real place of that field.  The situation is
simpler when $F$ contains $E$ (e.g. our running example of $F = N = \Q(i)$
where $E=\Q$) but we will not assume this is the case.

Writing $F = \Q[x]/(f)$ for an irreducible polynomial $f\in\Q[x]$ and
factoring $f$ in the extensions $N$ and $E$ of $\Q$ we see that the $E$-algebra
$A = F\otimes_\Q E$ is étale, we have $A \isom \prod_i E_i$ which we can
interpret both as an isomorphism of $F$-algebras and as an isomorphism of
$E$-algebras.  From the first point view composing each embedding
$F\embed E_i$ with the embedding $w\colon E_i\to \C$ we obtain an enumeration
of the archimedean places of $F$, with $E_i \isom N$ for complex places
(say those are indexed by $1\leq i\leq s$) and $E_i \isom E$ for real places
(say for $s<i\leq s+r+t$).  Without loss of generality we assume the inclusion
$F\embed E_i\embed \C$ is the place $v$ fixed above.

Now thinking of $A$ as a $E$-algebra let $\tGG = \GG \times_F A$ thought of as
an algebraic group over $E$.  Equivalently
$\tGG = \left(\Res^F_\Q \GG\right) \times_\Q E$ (Weil restriction of scalars).
The factorization of $A$ then gives a factorization $\tGG = \prod_i \tGG_i$
as groups over $E$, where $\tGG_i(E_w) \isom \GG(F_v)$ if $v$ is the $i$th
archimedean place of $F$ in terms of our enumeration above.  Having done this
we also write $G_i$ for $G_v$, especially in situations where we would like
$v$ to vary over finite places.

The embedding $F\embed A$ gives an inclusion $\GF \to \tGG(E)$. 

Finally let $p$ be a rational prime which splits completely in $N$ (hence
also in $F\subset N$).  Choosing any place $w_p\colon N\to\Qp$ lying over $p$
all embeddings of $F$ into $\Qp$ factor through $w_p$ (as above, with different
embeddings of $F$ into $N$). Restricting $w_p$ to $E$ we then have
$$E_i \otimes_E \Q_p \isom
\begin{cases} \Qp & E_i = E\\ \Qp\times\Qp & E_i=N\end{cases} $$
since when $E_i = N$ there are two places of $N$ lying over the place $w_p$
of $E$.  Furthermore, $\Gal(N/E) = \{ c_w,\id\}$ acts transitively on these
places so that $c_w$ swaps them.

Finally let $\gamma\in\GF$ and suppose that its image in $\tGG_1(E_w)$ lies
in the subgroup $g_1 \tGG_1(E_w)^{c_w} g_1^{-1}$ for some $g_1 \in \tGG_1(E)$
(we think of $\tGG_1(E_w)^{c_w}$ as the "standard" copy of $\SL_2(\R)$ in
$\tGG_1(E_w) \isom \SL_2(\C)$).
Since the image of $g_1^{-1}\gamma g_1 \in \tGG_1(E)$ is fixed by $c_w$,
this persists when we embed our group in $\Qp$ using the place $w_p$, and it
follows that the image of $\gamma$ in $\tGG_1(E_{w_p})$ lies in
$g_1 \tGG_1(E_{w_p})^{c_w} g_1^{-1}$, in other words in a particular conjugate
of the diagonal subgroup $\tGG_1(E_{w_p})^{c_w}$ (so-called because when we
identified $\tGG_1(E_{w_p})^{c_w}\isom\left(\SL_2(\Qp)\right)^2$ the Galois
automorphism $c_w$ acts by exchanging the factors).

Before going any further, it may be helpful to have two concrete examples
of the setup so far.
\begin{exa}
Let $F=\Q(\xi)$ where $\xi^3 = 2$ and let $D = M_2(F)$ be the matrix algebra
over F.  The field $F$ has one real place $v'$
(coming from the unique real cube root of $2$) and one complex place $v$
(coming from the two non-real cube roots), so the corresponding Lie group
is $\SL_2(\C)\times\SL_2(\R)$ with the symmetric space $\HH^{3}\times\HH^{2}$.
Let us now examine how this fits in
the framework above.

For this let $\omega$ be a cube root of unity so that
$N = F(\omega) = F(\sqrt{-3})$ is the normal closure of $F$ in which the
Galois conjugates of $\xi$ are $\xi,\xi\omega,\xi\omega^2$.

We choose the complex embedding $v$ so that\footnote{We use the mumber theory
convention $e(z) = \exp(\pi iz)$} $v(\xi) = \sqrt[3]{2} e(1/3)$ and
extend this to $N$ by $w(\omega) = e(1/3)$.  The Galois group $\Gal(N/\Q)$
is the full permutation group on the cube roots of $2$ in $N$;
among its elements let $c_w$ be the one fixing
$\xi \omega^2$ and exchanging $\xi,\xi\omega$.  
Note that $w(\xi\omega^2=\sqrt[3]{2} \in \C$ so $c_w$ is complex conjguation
at $w$, and one can also check that $c_w(\omega) = \omega^2$.  We thus
obtain the subfield $E = N^{c_w} = \Q(\xi\omega^2)\subset N$,
a field which is abstractly isomorphic to $F$ but disjoint from it as a
subfield of $N$.

Next, we have $F\otimes_\Q E = E[x]/(x^3-2) \isom E[x] / (x^2+\xi\omega^2 x + \xi^2\omega)(x-\xi\omega^2) \isom E_1 \oplus E_2$
where $E_1 \isom N = E(\omega)$ and $E_2 \isom E$ (and also $E_2 \isom F$!).
Extending scalars in $D$ we obtain the algebra
\begin{align*}
\tD & = D \otimes_F (E_1\oplus E_2) \\
    & \isom (D\otimes_F E_1) \oplus (D\otimes_F E_2)
    & \isom (D\otimes_F N) \oplus (D\otimes_F E)
\end{align*}

It follows that $\tD \isom M_2(N) \oplus M_2(E)$.  Restricting to elements
of determinant one produces from the algebras $D$, $\tD$ the $F$-group
$\GG(K) = \{ g \in M_2(K) \mid \det(g)=1\}$ and the $E$-group
$\tGG(K) = \{ (g_1,g_2)\in M_2(K[\omega])\times M_2(K) \mid \det(g_1)=\det(g_2)=1\}$
for any field extension $K$ of $F$, $E$ respectively.  It may seem odd to
distinguish between $E$ and $F$ (which are after all isomorphic fields in this
example), so let us put this distinction use.
We have $G_\infty = \SL_2(\C) \times \SL_2(\R)$, but also
$\tGG(E_w) \isom \SL_2(N_w) \times \SL_2(E_w) \isom \SL_2(\C) \times \SL_2(\R)$.
Furthermore the isomorphism of the two groups $G_\infty \isom \tGG(E_w)$
preserves the diagonal inclusion of $\SL_2(F)$ in both of them -- exactly
because the complex place $w$ of $N$ restricts to the complex
place of $F$ but the \emph{real} place of $E$.  Furthermore complex
conjugation in the first factor is now algebraic: it is the obvious
automorphism of the algebra $K[\omega] = K[x]/(x^2+x+1)$.

Thus let $p$ be a rational prime which splits completely in $N$.  Let
$v_1,v_2,v_3$ be the three places of $F$ lying over $p$, giving us the group
$G_p = \SL_2(F_{v_1}) \times \SL_2(F_{v_2}) \times \SL_2(F_{v_3})$
together with the diagonal embedding of $\GG(F)=\SL_2(F)$ in $\Gi\times G_p$.
In the proof of \cref{thm:main2} we will obtain elements $\gamma\in\GG(F)$ whose image in the factor
$G_v \isom \SL_2(\C)$ lie in $\SL_2(\R)$, that is are fixed by the complex
conjugation, and we will want
to choose elements of $G_p$ which "avoid" those $\gamma$ in some sense.
For this let $w_1$ be a place of $E$ lying over $p$.  Observe that since $p$
splits in $N$, $\Qp$ contains the cube roots of unity,
$E_{w_1}[\omega] \isom \Qp\oplus \Qp$ and this isomorphism gives us the other
two places of $F$ in
$$\tGG(E_{w_1}) \isom \SL_2(\Qp\oplus\Qp)\times \SL_2(\Qp) \isom G_p\,.$$
However, all this is compatible with the action of $c_w$ in the first factor,
mapping $\omega$ to $\omega^2=-1-\omega$, which amounts to exchanging the two
copies of $\Qp$.  Accordingly let $\gamma \in \GF$ be such that $v(\gamma)$
is real.  Then the image of $\gamma$ in $\tGG(E_w)$ is fixed by $c_w$,
so the same must hold for the image in $\tGG(E)$ and hence the image
in $\tGG(E_{w_1})$.  In short, we have chosen the factors in the isomorphism
$G_p \isom (\SL_2(\Qp))^3$ so that those $\gamma$ which are real in the first
complex place are diagonal (have the same image) in the first two $p$-adic
places.  We could then "avoid" (in the precise sense we need) such
$\gamma$ by choosing Hecke operators (see below) living at exactly one of
the two places exchanged by $c_w$.
\end{exa}

\begin{exa}Continuing with the same field $F$, let
$D=\left(\dfrac{-\xi,-\xi}{F}\right)$ be the quaternion algebra with
$F$-basis $1,i,j,k$ such that $i^2=j^2=-\xi$ and $ij=-ji=k$.  The reduced
norm of this algebra is then the quadratic form
\begin{equation}\label{normform}
\det(a+bi+cj+dk)=a^2+\xi b^2+\xi c^2+\xi^2d^2
\end{equation}
which is positive definite in the real embedding $v'$ for which
$v'(\xi)=\sqrt[3]{2}$.  Then $D_{v'}$ is a noncommutative real division algebra,
in other words is isomorphic to Hamilton quaternions -- whereas we still have
$D_v\isom M_2(\C)$ since the only complex division algebra is $\C$.

Viewing the quaternions as a 2-dimensional vector space over $\C$ and letting the
invertible quaternions act on themselves by multiplication shows that the group
of norm-$1$ quaternions isomorphic to $\SU(2)$, and hence for the group $\GG$
of norm-$1$ elements in $D$ we have
$$
\Gi\isom \SL_2(\C)\times\SU(2)
$$
With the corresponding symmetric space $\HH^{3}$.  The rest of the discussion
in the previous example continues unchanged: for any prime $p$ we can still
realize $G_p \isom \tGG(E_{w_1})$ so that complex conjugation in the first
factor of $\tGG$ corresponds to both complex conjugation in $G_v$ and to
swapping the first two factors in $G_p \isom \left(\SL_2(\Qp)\right)^3$.

This elucidates the extra difficulty of proving our results in the case
of lattices such as $\SL_2(\OF)$ which do not act cocompactly on $\HH^{3}$ but
act on it in a more complicated fashion than a Bianchi group such as $\SL_2(\Z[i])$.
Consider an element of $\SL_2(F)$ which is real in the complex embedding.
When $F=\Z[i]$ (say) complex conjugation is a Galois automorphism of $F$ and
it is clear how it acts on $\SL_2(F_{v_1})\times\SL_2(F_{v_2})$ for the two
places of $F$ lying over a rational prime $p$ splitting in $F$; elements
of $\SL_2(\Q(i))$ which are real in the complex embedding lie in $\SL_2(\Q)$
and clearly embed diagonally.  
On the other hand in the present example $F$ has no Galois automorphism and
so showing that elements of $D^1$ that are real at the complex place
have identical images at two of the three $p$-adic places most naturally
involves going beyond $F$.
\end{exa}

\subsection{The real group}\label{sec:RealGp} 
We return to the isomorphism
$\Gi \isom \left(\SL_2(\C)\right)^s \left(\SL_2(\R)\right)^r \left(\SU(2)\right)^t$
and fix some subgroups of this group.

At each infinite $v$ where $D_v$ splits, let $A_v \subset G_v$ be the group 
corresponding to the group of diagonal matrices with positive real entries
under the isomorphism above.  Also let $K_v$ be a compatible maximal compact
subgroup (corresponding to the subgroup $\SU(2)$ at a complex place, $\SO(2)$
at a real place).  From these let $M_v = Z_{K_v}(A_v)$ be the centralizer of
$A_v$ in $K_v$, so that $M_v$ is the group $\{\pm I\} \subset \SL_2(\R)$ at
a real place or the group $U(1) \subset \SL_2(\C)$ of diagonal matrices 
with inverse entries both of modulus $1$.  Observe that in either case the group
$A_v M_v$ consists of the $F_v$-points of a maximal $F_v$-split algebraic
torus of $\GG\times_F F_v$.

At the real places $v$ where $D_v$ remains a division algebra 
$K_v = G_v \isom \SU(2)$ is a maximal subgroup.  With this choice
$K = K_\infty = \prod_{v\mid\infty} K_v$ is a maximal compact subgroup of
$G$ and $G/K \isom \left(\HH^{(3)}\right)^s \left(\HH^{(2)}\right)^r$.

We fix (arbitrarily) once and for all a left-invariant Riemannian metric on
$G$ (equivalently, a positive definite quadratic form on $\Lie G$).
This induces a left-invariant metric on $X=\Gamma\bs G$ where the quotient
map is non-expansive.  For a subset $U\subset G$ we write $\Ue$ for its
$\epsilon$-neighborhood with respect to this metric.  This metric is used
in \cref{sec:homogeneous,sec:proofs}, but in both cases we can first restrict
our attention to compact subsets of $G$.  The choice of metric thus affects
some overall constants but not the bottom line.  For example in 
\cref{sec:homogeneous} we establish bounds of the form
$\mu(\Ue) \leq C \epsilon^h$; with a different metric $\Ue$
would be contained in $U_{c\epsilon}$ with respect to the new one and the
bound would be identical except for the value of the constant $C$.

\subsection{$p$-adic and adelic groups}\label{sec:adelic-quotient}
Let $R\subset D$ be an order, that is a subring which is an $\OF$-lattice in
$D$.  Then for every finite place $v$ of $F$,
$R_v = R\otimes_{\OF} \Ov$ is an order in $D_v$,
that is a compact open $\Ov$-subalgebra of $D_v$.
Its group of units $R_v^1$ is then a compact subgroup of $G_v$ which
is a maximal compact subgroup $K_v$ at almost all places.

For all but finitely many places, $D_v \isom M_2(F_v)$ and 
then $G_v\isom\SL_2(F_v)$ and $K_v \isom \SL_2(\Ov)$
(we don't choose $K_v$ at the finitely many places where $D_v$ is a division
algebra or where $R_v$ is not a maximal order).

Let $\GAf = \prod'_{v<\infty} (G_v:K_v)$ be the restricted direct product
of the $G_v$ and let $\GA = \Gi \times \GAf$.  Then the diagonal embedding
$\GF \embed \GA$ realizes $\GF$ as a lattice there.  Fix an open compact
subgroup $\Kf\subset \GAf$.  Then there exists a finite set $S$ of
finite places (including all places where $K_v$ was left undefined above)
such that $\Kf = H \times \prod_{S\not\ni v<\infty} K_v$ for some open
compact subgroup $H<\prod_{v\in S} G_v$.
We will generally ignore all places in $S$.

Let $\Gamma = \GF \cap \Kf$.  Its image in $\Gi$ is a lattice, and we have
the identification\footnote{We use here that $\GG$ is a form of the simply
connected algebraic group $\SL_2$; in general the adelic quotient would
correspond to a disjoint union of quotients $\Gamma\bs G$}
$$X = \Gamma\bs G \isom \GF \bs \GA / \Kf$$
given by mapping the coset $\Gamma \gi$ to the double coset
$[\gi,1] = \GF (\gi,1) \Kf$ (here $1$ is the unit element of $\Af$).

A congruence subgroup of $G$ is one that contains a lattice $\Gamma$ as
above as a finite-index subgroup.  Since equidistribution modulo $\Gamma$
implies equidistribution modulo any lattice containing it (assuming the eigenfunctions
are properly invariant) we assume the lattice has this form without loss
of generality.

We further set $Y = X/K = \Gamma \bs S$ where $S=G/K$ is the symmetric
space of $G$.  Equipping $X$ with the $G$-invariant probability measure
we study functions ("automorphic forms") in $L^2(X)$, and identify $L^2(Y)$
with the subset of $K$-invariant functions.  For each noncompact factor
$G_i$ of $G$ let $\omega_i$ be the Casimir element in the universal enveloping
algebra $U(\lieg_i\otimes_\R \C)$.  Then $\omega_i$ acts on the right on
smooth functions on $G$ and $X$, equivariantly with respect to the right
$G$-actions.  It thus descends to a differential operator on functions on $S$
where it coincides with the Laplace--Beltrami operator on the irreducible
symmetric space $G_i/K_i$. A \emph{Maass form} on $Y$ is a function
$\psi\in L^2(Y)$ which is a joint eigenfunctions of the $\omega_i$.

\subsection{Hecke operators}\label{sec:Hecke} 
For every finite place $v$, the convolution algebra of locally constant
compactly supported functions on the totally disconnected group $G_v$
acts on the right by convolution on the space of smooth functions on
$\GF\bs\GA$.  At a place $v$ where $K_v$ is defined and contained in $\Kf$,
the subalgebra of bi-$K_v$-invariant functions
$\calH_v = \calH(G_v:K_v) = \Cic(K_v \bs G_v/K_v)$ preserves the subspace
of right-$\Kf$ invariant functions, and hence acts on the space of functions
on $\GF\bs\GA/\Kf = \Gamma\bs G$.  We note that the actions of the different
$\calH_v$ commute with each other and with the right $G$-action on our space,
and call the algebra of operators on functions generated by all of them the 
\emph{Hecke algebra}.

Let $\calP$ be the set of rational primes $p$ which split completely in $N$,
and such that every place $v$ of $F$ above $p$ has $K_v$ as a factor of $\Kf$,
a set of positive natural density in the primes. Those are the
primes whose Artin symbol in
$\Gal(f) = \Gal(N/\Q)$ is trivial and the claim follows immediately
from the effective Chebotarev Density Theorem \cite[Theorem 1.2]{PlatonovRapinchuk:AlgGps_NT}.  
For such $p$ let $\calH_p = \otimes_{v\mid p} \calH_v$ be the
"Hecke algebra at $p$".  We will only consider Hecke operators in the
restricted Hecke algebra $\calH = \otimes'_{p\in \calP} \calH_p$ generated
by the $\calH_p$.

Since the right actions of $G_v$ and $G_i$ on $\GF\bs\GAF$ commute, the
Hecke operators commute with the differential operators of the previous section.
A \emph{Hecke--Maass form} is a Maass form $\psi\in L^2(Y)$ which is also a joint
eigenfunction of the Hecke algebra.  Since the group actions commute it also
follows that if $\phi\in L^2(X)$ is any other element of the irreducible
representation generated by $\psi$ then $\phi$ is also an eigenfunction of
the Hecke algebra with the same eigenvalues as $\psi$.

\section{Homogeneity}\label{sec:homogeneous}
In this section we invoke the necessary machinery to deduce
\cref{thm:main1} from \cref{thm:main2}.  The main new ingredient
(which was already known to experts) is the statement of a so-called
``diophantine lemma'' (\cref{prop:dioph} below) for groups defined over
a number field.

As described in the introduction let $\{\psi_j\}_{j\geq1} \subset L^2(Y)$ be
a normalized sequence of Hecke--Maass forms with Laplace eigenvalues tending
to infinity.  Let $\mb_j$ be the corresponding probability measures on $Y$
(recall that those are the measures with density $\abs{\psi_j(x)}^2$ with
respect to the Riemannian measure); our ultimate goal is to show that the
$\mb_j$ converge to the normalized Riemannian volume on $Y$, or equivalently
that this is the only subsequential limit.  Accordingly (passing to a subsequence)
we assume the $\mb_j$ converge weak-* to a measure $\mb$ on $Y$. 
The main result of  \cite{Zaman:EscapeThesis} is that any such limiting measure $\mb$ on $Y$ 
is a probability measure, even if $Y$ non-compact. 

Again passing to a subsequence there is an infinite place $v$ such $G_v$
is non-compact and such that the Laplace--Beltrami eigenvalues of $\psi_j$
with respect to the Laplace operator at $v$ tend to infinity; without loss
of generality we may assume it is the first place in our enumeration.
Then by the microlocal lift of
\cite{SilbermanVenkatesh:SQUE_Lift} there are Hecke eigenfunctions
$\phi_j\in L^2(X)$ such that any subsequential weak-* limit $\mu$ of the
associated measures $\mu_j$ has the following properties:

\begin{enumerate}
\item $\mu$ projects to $\mb$ under the map $X\to Y$ (in particular, $\mu$ is
      a probability measure).
\item $\mu$ is $A_1$-invariant.
\end{enumerate}

Passing to a subsequence yet again we may assume that the $\mu_j$ themselves
converge, and would like to show that the limit $\mu$ is the $G$-invariant
probability measure on $X$.  At this point the differential operators exit
the stage: in the sequel we only use the fact that $\phi_j$ are normalized
Hecke eigenfunctions on $X$ and that $\mu$ has the two properties above.

The rest of the section is divided as follows: in \cref{sec:posent}
we show that every ergodic component of $\mu$ has positive entropy with
respect to the action of $A_1$.  In \cref{subsec:rigidity} we
then invoke a measure rigidity theorem of Einsiedler--Lindenstrauss and interpret its results, 
classifying the possible ergodic components of
$\mu$.

\subsection{Positive Entropy}\label{sec:posent}

Let $a\in A_1$ be non-trivial.  Under the isomorphism $G_1 \isom \SL_2(F_1)$,
$a$ is a diagonal matrix with distinct real positive entries,
so $T_1 = Z_{G_1}(a)$ is the group $A_1 M_1$ of all diagonal matrices in $G_1$.
Setting $\Ggt = \prod_{i\geq 2} G_i$, the centralizer of $a$ in $G$ is then
$T = T_1 \cdot \prod_{i\geq 2} G_i$. 

For a compact neighborhood of the identity $U\subset T$ recall our notation
$\Ue$ for an $\epsilon$-neighborhood of $U$ in $G$.  We will
establish the following result:

\begin{prop}There is a constant $h>0$ such that for any compact subset
$\Omega \subset G$ (expected to be large) and any $U$, we have for any
$\epsilon$ small enough (depending on $\Omega$, $U$) and any Hecke
eigenfunction $\phi_j\in L^2(X)$ that for all $g\in \Omega$,
$$\mu_j\left(\Gamma g\Ue \right) \ll_{\Omega,U} \epsilon^h\,.$$
\end{prop}

The key point is that the implied constant is independent of $\phi$, so that
the limiting measure $\mu$ satisfies the same inequality.

This result is essentially contained in \cite{SilbermanVenkatesh:AQUE_Ent}
(on some level already in \cite{LindenstraussBourgain:SL2_Ent}) except
that \cite{SilbermanVenkatesh:AQUE_Ent} assumes that $G$ is $\R$-split,
which is the not the case for $\SL_2(\C)$.  In fact all that is needed there
is that the centralizer $T_v$ at some infinite place $G_v$ is a torus.

\subsubsection{Diophantine Lemma}\label{sec:dioph}
We begin by reviewing a notion of "denominator" for elements of $F$ and $\GF$.
The construction is the natural generalization to number fields
of the notion used in \cite{SilbermanVenkatesh:AQUE_Ent} for groups over the
rationals.  For further discussion see \cref{justify-dioph-fields}.

\begin{defn}The \emph{denominator} of $x_v\in F_v$ is the natural number
$\denom_v(x_v) = \max\{ \abs{x_v}_v, 1\}$.  Equivalently if $x_v = \varpi_v^k y$
with $y\in\Ov^{\times}$ then
$$\denom_v(x_v)=\begin{cases} q_v^{-k} & k \leq 0 \\ 1 & k \geq 0\end{cases}\,.$$

We now extend this definition.  First, for $x\in\Af$ or $x\in\AF$ we set
$$\denom{x} = \prod_{v<\infty} \denom_v(x_v)\,,$$
where all but finitely many of the factors are $1$ since $x_v \in \Ov$ for
all but finitely many $v$.  Second for $x\in F$ let $\denom(x)$ be the
denominator of its image in $\AF$.
\end{defn}

We further extend the definition to matrix algebras over the above rings.
Specifically for $x_v\in M_N(F_v)$ let $\denom_v(x_v)$ be the largest of the
denominators of the matrix entries (equivalently this is the denominator
of the fractional ideal they generate), and again extend this to $M_N(\Af)$,
$M_N(\AF)$ and $M_N(F)$ by multiplying over the places and restriction,
respectively.

Finally, the \emph{product formula} $\prod_v \abs{x}_v = 1$ for $x\in F^\times$
implies 
$$
(\prod_{v\mid \infty} \abs{x}_v)\cdot \denom{x} \geq 1,
$$ 
for all nonzero
$x\in F$, and hence also for all $x\in GL_n(F)$.

\begin{rem}
We could have defined the denominator of $x\in M_N(\Af)$
by taking the largest of the denominators of its entries (call that $\denom'$
for the nonce); the two notions are equivalent in that
$\denom'(x) \leq \denom(x) \leq \denom'(x)^{n^2}$ for all $x$.  Our choice
agrees with defining for $x\in M_N(F)$ the denominator as the denominator
of the fractional ideal generated by the matrix entries.  In the sequel
the precise choice of denominator affects the exponents in
\cref{prop:dioph} and thus the precise entropy $h$ we obtain in
\cref{prop:posent} but does
not change the \emph{positivity} of $h$, which suffices for our
purposes (and ultimately by determining the limit exactly we prove
the measure has maximal entropy anyway).
\end{rem}

The following is an immediate calculation and we omit the proof.
\begin{lem} Let $x_v,y_v\in M_N(F_v)$. 
Then
$$\denom_v(x_v+y_v),\denom_v(x_v y_v) \leq \denom_v(x_v)\denom_v(y_v)\,,$$
and in particular if $y_v \in \GL_N(\Ov)$ then
$\denom_v(x_v y_v) = \denom_v(x_v)$.
Furthermore there is a constant $C$ depending only on $n$ such that
$$\denom_v(x_v^{-1}) \leq \denom_v(x_v)^C$$
for $x\in\SL_N(F_v)$.
\end{lem}
\begin{cor}Multiplying place-by-place the same inequalities hold for
the denominators in $M_N(\Af)$ and $M_N(F)$.
\end{cor}

Finally if $\GG$ is a linear algebraic $F$-group fixing an $F$-embedding
$\rho\colon\GG\to\SL_N$ allows us to define the denominators of elements
of $\GG(F_v)$, $\GA$, $\GF$.  The same reasoning as above would show that
changing the embedding gives an equivalent definition in the sense above
(i.e. up to multiplying by constants and raising to powers).  In our case
letting $D$ act on itself by multiplication gives an embedding $G\to\SL_4(F)$
and using the order $R$ to define the integral structure we further have
$\rho(K_v) \subset \SL_4(\Ov)$ whenever $K_v$ is defined.  It follows
that our local denominator is a bi-$K_v$-invariant function on $G_v$, so
every basic Hecke operator (the characteristic function of a double coset
$K_v g_v K_v$) has a well-defined denominator.  Furthermore (identifying
$G_v \isom \SL_2(F_v)$) this double coset has a representative
$a_v = \begin{pmatrix} \varpi_v^m & 0\\ 0 & \varpi_v^{-m}\end{pmatrix}$
for some $m\geq 0$ in which case we call $2m$ the \emph{radius} of the
Hecke operator; its denominator is then $q_v^m$.

Given constants $c_1,c_2$ the set of \emph{potential Hecke operators} is
the set of $\gf\in\GAF$ 
of denominators
at most $c_1 \epsilon^{-c_2}$.  We will be using Hecke operators of
uniformly bounded radius so the main effect here is to bound the set of places
$v$ under consideration.  Say that a potential Hecke operator
causes an \emph{intersection} at $\gi\in\Omega$ if there is $\gamma\in\GF$
such that
$$\gamma \gi \Ue\Kf \cap \gi \Ue\gf\Kf \neq \emptyset\,,$$
in which case we say $\gamma$ is \emph{involved} in the intersection.

\begin{prop}\label{prop:dioph} One can choose $c_1,c_2,\epsilon_0>0$ 
such that if $\epsilon<\epsilon_0$ then for all $\gi\in\Omega$ there is an
algebraic $F$-torus $\bSS \subset \GG$ such that all $\gamma\in\GF$ involved
in intersections at $\gi$ lie in $\bSS(F)$.
\end{prop}

\begin{proof}
The first step of the proof is to show that there is a choice of $c_1,c_2,\epsilon_0$ so that for any $\gi\in\Omega$, all $\gamma\in\GF$ involved in intersections at $\gi$ commute with each other, so that the $F$-group generated by them is commutative.   
In fact we simply take $c_1=1$ and show that there is a choice of $c_2, \epsilon$ as above.  
For this, suppose we have $b_1,b_2 \in \Ue$ and $\kf\in\Kf$ such that
$\gamma \gi b_1 = \gi b_2 \gf \kf$, or equivalently
$$\gamma = \gi b_2 b_1^{-1} \gi^{-1} \gf \kf.$$
By hypothesis
$\denom(\gamma) = \denom(\gf\kf) = \denom(\gf) \leq  \epsilon^{-c_2}$.
In addition since $\Omega$ is compact we have that $\gi b_2 b_1^{-1} \gi^{-1}$
is $O_\Omega(\epsilon)$-close to an element of $U U^{-1}$.

The matrix commutator
$[\gamma_1,\gamma_2] = \gamma_1 \gamma_2 - \gamma_2 \gamma_1$
(interpreted via the fixed embedding in $SL_4$)
is a polynomial function $\GG^2\to\SL_4$ with coefficients in $F$.
Thus we have a constant $c$ so that if $\gamma_1,\gamma_2\in\GF$ then
$\denom([\gamma_1,\gamma_2]) \leq c \denom(\gamma_1)^c \denom(\gamma_2)^c$.
Now suppose that $\gamma_1,\gamma_2$ are both involved in intersections at
$\gi$.  Then $\gamma_1$ and $\gamma_2$ are $C_\Omega\epsilon$-close to 
elements of $\gi U U^{-1} \gi^{-1}$.
Since the commutator is a smooth function on
$G\times G$ the element $[\gamma_1,\gamma_2]\in G$ is
$O_{\Omega,U}(\epsilon)$-close to the commutator of two elements drawn from
$\gi U U^{-1} \gi^{-1}$.  Since $T_1$ is commutative (recall that we defined $T_1=Z_{G_1}(a)=A_1M_1$, that is the group of diagonal matrices in $G_1$) we conclude that
$\abs{[\gamma_1,\gamma_2]}_1 = O_{\Omega,U}(\epsilon)$, and that for $i\geq 2$
we have $\abs{[\gamma_1,\gamma_2]}_i = O_{\Omega,U}(1)$ if we also assume
$\epsilon<1$.

Suppose $\gamma_1,\gamma_2$ do not commute.  By the product formula we then
have
$$\denom([\gamma_1,\gamma_2]) \cdot \prod_i \abs{[\gamma_1,\gamma_2]}_i \geq 1\,,$$
that is 
\begin{equation}\label{impossible}
c \epsilon^{-c c_2} O_{\Omega,U}(\epsilon) \cdot O_{\Omega,U}(1)^{r+s+t-1} \geq 1\,.
\end{equation}
If $c_2 < \frac1c$ then \eqref{impossible} is impossible for any $\epsilon$ small enough. We conclude that there are $c_2,\epsilon_0>0$ such that the $\gamma$ that are involved in intersections commute, so the
$F$-subgroup of $\GG$ generated by those $\gamma$ is commutative.  If this
subgroup is finite its elements are semisimple (and contained in a torus),
and otherwise the only commutative connected subgroups of $\SL_2$ 
(even over its algebraic closure) are either tori or unipotent.  Note
that tori are self-centralizing whereas the component group of the centralizer
of a unipotent subgroup is represented by the center of $\SL_2$.

This concludes the argument when $D$ is a division algebra, since in that case
$\GG(F)$ consists entirely of semisimple elements and the maximal commutative
$F$-subgroups are all tori.  When $D=M_2(F)$ we need to rule out the
possibility that some $\gamma$ causing an intersection is unipotent. The basic idea for this is that a unipotent element that is close to a torus is close to the identity element. Thus if $\gamma$ is unipotent then the bound for the denominator of $\gamma$ forces it to be the identity element. 
To prove it, choose $U$ and $\epsilon_0$ small enough so that every
$g \in U_{\epsilon_0}U_{\epsilon_0}^{-1}$ is close enough to the identity
to have $\abs{\Tr(g)-2}_i < 1$ at each infinite place.
Now suppose that $\gamma\in\GF$ is involved in an intersection at $\gi$,
and that $\gamma$ is unipotent.  Then $\abs{\Tr(\gamma)-2}_1<1$ so
$\Tr(\gamma)=2$ (the alternative was that $\Tr(\gamma)=-2$).

Writing $\gamma = \gi b_2 b_1^{-1} \gi^{-1}$ we get that
$\Tr(b_2 b_1^{-1})=2$.  At the factor $G_1$
each $b_2b_1^{-1}$ is $\epsilon^\eta$-close to an element of $T_1$, for some absolute $\eta>0$, so this element
has trace $2+O(\epsilon^\eta)$ (as measured by $\abs{\cdot}_1$); since $T_1$ is
a (fixed) torus this element is $O(\epsilon^\eta)$-close to the identity element. Finally,
since $\Omega$ is compact, conjugation by $\gi$ does not change this fact:
the image of $\gamma$ itself in $G_1$ is $O(\epsilon^\eta)$-close to the identity element.
Thus $\gamma = \Id$ if $c_2<\eta$ and
$\epsilon$ is small enough. 

\end{proof}

\begin{rem}\label{justify-dioph-fields} The concrete argument using the
commutator already appears in \cite[Lem.\ 3.3]{LindenstraussBourgain:SL2_Ent}
and two more general versions appears in
\cite[Sec.\ 4]{SilbermanVenkatesh:AQUE_Ent} for group over $\Q$.  The new
observation here is that one only needs information about the $\gamma$
at a single real place (assuming boundedness at the other real places).
In particular we cannot just restrict scalars to $\Q$ and apply the earlier
result.
\end{rem}

\subsubsection{Bounds on the mass of tubes}\label{subsec:tubemass}
The arguments of \cite[\S5]{SilbermanVenkatesh:AQUE_Ent} (in fact already
of \cite{LindenstraussBourgain:SL2_Ent} for the group at hand) now establish
the following, using the diophantine lemma we proved above instead of the
analogous lemmas in those papers, noting that our notion of denominator
is equivalent to the notion we'd obtain over $\Q$ for the group
$\Res^F_\Q \GG$.

\begin{prop}\label{prop:posent}
Fix compact subsets $\Omega\subset G$, $C\subset T_1 \times \Ggt$.
Then for any normalized Hecke eigenfunction $\phi\in L^2(X)$
(pulled back to a function on $G$) and any $\gi\in \Omega$ we have 
$$\int_{\gi \Ue} \abs{\phi(x)}^2 \dvol(x) = O(\epsilon^h)$$
for a universal constant $h>0$ where the implied constant depends on
$\Omega, C$ but not on $\phi$.
\end{prop}

\begin{cor}\label{posent1} $A_1$ acts on almost every ergodic component of $\mu$ with positive
entropy.
\end{cor}
The fact that \cref{prop:posent} implies \cref{posent1} is standard. 
Roughly speaking, it follows by the (relative) Shannon--McMillan--Brieman Theorem and the fact that the shape of a generic atom in a refinement of an appropriate partition of $X$ by the action of $A_1$ is approximated by sets $U_\epsilon$ as in \cref{prop:posent}. For the basic definitions of entropy and precise statements from which the fact that positive entropy of a.e. ergodic component follows from bounds as in \cref{prop:posent}
see the appendix and Section 4.1 in \cite{ShemTov:OnePlaceHighRank}. 

\subsection{Measure rigidity}\label{subsec:rigidity}
In this section we interpret $\mu$ as a measure on $\GF\bs\GA$.  Let $v$ be
any finite place at which $G_v$ is non-compact.  Then
\cite[Thm.\ 8.1]{Lindenstrauss:SL2_QUE} shows that $\mu$ is
$G_v$-\emph{recurrent}: if $B \subset \GF\bs\GA$ is any set of positive
measure then for almost every $x\in B$ the set of returns
$\{g_v \in G_v \mid x g_v \in B\}$ is unbounded.

\begin{lem} For $\mu$-almost every $x\in \GF\bs\GA$ the group
$$
S(x)=\{ h \in M_1 \times \Ggt \times G_v \mid x h = h\}
$$
is finite.
\end{lem}

\begin{proof}
Let $x = \GF g$ with $g=(\gi,\gf)\in\GA$ where 
$g_\textrm{f} \in \Kf$, and $h\in \GA$ with $h_1 \in M_1$, and suppose that $\GF g\cdot h = \GF g$.
Then there exists $\gamma \in \GF$ such that $g\cdot h = \gamma g$ or
equivalently 
$$ \gamma =  ghg^{-1}\,.$$
Thus conjugation by $g$ (and the inclusion $\GF\embed G_1$)
embeds $S(x)$ in $\GF\cap g_1 (A_1 M_1) g_1^{-1}$.
As in \cref{sec:dioph} above this
intersection consists of the $F_1$-points of the diagonal $F_1$-torus of $G_1$,
which is one-dimensional, so if the group in the statement
was infinite it would be Zariski-dense in the conjugate
torus; equivalently the Zariski closure of this group would be an
$F$-torus $\TT\subset\GG$ with $g_1^{-1} \TT g_1$ diagonal.

For each torus $\TT$ the set of $g_1 \in G_1$ that diagonalize it
lies in two $A_1 M_1$-cosets (due to the effect of the Weyl group).
Since there are countably many such tori the set of $x$ for which $S(x)$ is infinite is contained in the countable union of images in
$\GF\bs\GA$ of the sets of the form
$$(g_1 A_k M_k) \times \Ggt \times \Kf\,,$$
and we need to show this is a null set.  Since our measure is
right-$\Kf$-invariant we may instead consider the image of the set
in $X=\Gamma\bs\Gi$, and it remains to recall that the positive entropy
argument showed that the images of the sets $(g_1 A_k M_k) \times \Ggt$,
have $\mu$-measure zero in a strong quantitative form (the argument gave
a uniform bound for the mass of $\epsilon$-neighborhoods of compact parts of
such images).
\end{proof}

At this point we have verified the hypotheses
of the following measure rigidity result.

\begin{defn} A measure $\nu$ on $X$ is \emph{homogeneous} if it is the unique
$H$-invariant probability measure on a closed $H$-orbit for a closed subgroup
$H<G$.
\end{defn}

\begin{thm}[Einsiedler--Lindenstrauss {\cite[Thm.\ 1.5]{EinsiedlerLindenstrauss:GeneralLowEntropy}}]\label{thm:EL15}
Let $\mu$ be an $A_1$-invariant probability measure on $X$ such that

\begin{enumerate}
\item $\mu$ has positive entropy on \ae ergodic component with respect to the action of $A_1$. \label{ent_cond}

\item $\mu$ is $G_v$-recurrent. \label{rec_cond}

\item\label{third_cond}
For $\mu$-\ae $z\in \Gamma_p\bs G\times G_p$ the group 
$$ \{h\in M_1\times \Ggt \times G_v\mid zh=z\}$$
is finite. 
\end{enumerate}
Then $\mu$ is a convex combination of homogeneous measures.  Furthermore,
for each such component $\nu$ the associated group $H$ contains a semisimple
algebraic subgroup of $G_1$ of real rank $1$ which further contains $A_1$
and conversely $H$ is a finite-index subgroup of an algebraic subgroup of $G$
(here we think of $G_1$ and $G$ as real algebraic groups).
\end{thm}

By strong approximation the lattice $\Gamma$ is irreducible in $G$, so that any
$G_1$-invariant measure is in fact $G$-invariant.  Thus every component $\nu$
for which $H$ contains $G_1$ is the $G$-invariant measure.  In particular
this happens when $G_1 \isom \SL_2(\R)$ (the corresponding place of $F$
is real) since $\SL_2(\R)$ has no proper semisimple subgroups, at which point
\cref{thm:main1} follows directly.  For the remainder of the paper we will thus
assume that $G_1 \isom \SL_2(\C)$ (the place is complex) and that $H$ contains
a conjugate of $\SL_2(\R)$ there
(these being the only proper semisimple subgroups).  

We now fix a particular representative for this conjugacy class.  Let
$\tHH_1 \subset \tGG_1$ be the fixed points of the automorphism $c_w$.
Indeed the isomorphism $G_1\isom\SL_2(\C)$ has $c_w$ act via complex
conjugation so $H_1 = \tHH_1(E_w)$ is is exactly $SL_2(\R)$, the subgroup
of matrices fixed by complex conjugation.  Recall that our choice of $A_1$,
the positive diagonal subgroup makes it a subgroup of this $H_1$, in fact
a real Cartan subgroup there.

\begin{prop}\label{prop:Horbits}  Let $\beta$ be a component of $\mu$ as in the
Theorem.  Then the support of $\beta$ is contained in the image in $X$ of a
submanifold of $G$ of the form
$$
L=g_1 H_1 M_1 \Ggt
$$ 
where $g_1\in \tGG_1(\Qbar\cap E_w)$.
\end{prop}

\begin{proof}
Let $\beta$ be a component of $\mu$ as in the theorem, supported on an $H$-orbit
$xH \subset X$ where $H$ contains a conjugate of the fixed group $H_1$.
Let $\tilde{H} \subset G = \tGG(E_w)$ be the semisimple $E_w\isom\R$-algebraic
subgroup of which $H$ is a finite-index subgroup.  Note that
$\tilde{H}\cap G_1 = H\cap G_1$ because $\SL_2(\R)$ is a maximal proper
closed subgroup of $\SL_2(\C)$, and for the same reason $H_1$ is the projection
to $G_1$ of any closed subgroup of $G$ that contains $H_1$.

Writing $x = \Gamma g$ for some $g\in G$ the translate $xHg^{-1}$ is the
$H'=g H g^{-1}$-orbit of the identity coset and supports an $H'$-invariant measure;
in other words $\Gamma \cap H'$ is a lattice in $H'$.  By the Borel Density
Theorem this lattice is Zariski-dense in $\tilde{H}'$ making $\tilde{H}$
defined over $E$ in view of $\Gamma \subset \GG(F) \subset \tGG(E)$.

The intersection $G_1\cap \tilde{H}' = g_1 (G_1\cap H) g_1^{-1}$ is also
defined over $E$, where we write $g_1\in\tGG_1(E_w)$ for the first coordinate of
$g$.  The group $g_1 (G_1 \cap H) g_1^{-1}$ and the fixed group $H_1$ are then
conjugate subgroups of $G_1$ defined over $E$, so by the Lemma below they
are conjugate by an element of $\tGG_1(E_w\cap \Qbar)$; 
let $g_1'$ be such an element.  We then have

$$G_1 \cap H = (g_1^{-1} g_1') H_1 (g_1^{-1} g_1')^{-1}\,.$$

Conjugating $A_1 < H$ (an avatar of the $A_1$-invariance of $\mu$),
gives a Cartan subgroup $(g_1^{-1} g_1')^{-1} A_1 (g_1^{-1} g_1') \subset H_1$,
and since all Cartan subgroups of a semisimple Lie group are conjugate
we obtain $h_1 \in H_1$ such that
$$ h_1 A_1 h_1^{-1} = (g_1^{-1} g_1')^{-1} A_1 (g_1^{-1} g_1')^{-1}\,.$$
so that
$h_1^{-1} (g_1^{-1} g_1')^{-1} \in N_{G_1}(A_1) = N_{H_1}(A_1) M_1$
and thus that $(g_1^{-1} g_1')^{-1} \in H_1 M_1$.

Finally in terms of all those elements,

\begin{align*}\Gamma g H
& \subset \Gamma g_1 (g_1^{-1} g) (G_1 \cap H) \Ggt \\
&= \Gamma g_1 (g_1^{-1} g_1') H_1 (g_1^{-1} g_1')^{-1} \Ggt \\
&\subset \Gamma g_1' H_1 H_1 M \Ggt \\
&= \Gamma g_1' H_1 M \Ggt\,.
\end{align*}
\end{proof}

Now there are countably many submanifolds of the form above, so to prove
our main theorem and show that $\mu$ is the $G$-invariant measure it suffices
to rule out each submanifold separately.

\begin{lem}\label{lem:algebraicconjugacy} Let $\GG$ be an algebraic group over $\R\cap\bar{\QQ}$.
Suppose that $\HH_1,\HH_2$ are two subgroups of $\GG$ and $\HH_1(\R)$ is contained in
$g\HH_2(\R)g^{-1}$ for some $g\in\GG(\R)$. Then $\HH_1(\R)$ is contained in $g'\HH_2(\R)g'^{-1}$ for some $g'\in \GG(\R\cap\bar{\QQ})$.
\end{lem}

\begin{proof}
The statement ``there exists $g\in\GG$ such that $g\HH_1 g^{-1}\subset \HH_2$''
can be expressed in first-order logic in the
language of fields. Thus the result follows from the fact that the first order theory of real closed fields is complete.  
\end{proof}

\section{Submanifolds with small stabilizers}\label{sec:smallness}
In the previous section we showed that components of the limit measure
$\mu$ other than the uniform measure are supported in images in
$X=\Gamma\bs\Gi$ of submanifolds of $\Gi$ of the form
$L=g_1 H_1 M_1 \Ggt$.  In this section we will use elementary arguments to
show that such manifolds $L$ and their submanifolds cannot be left-invariant
by subgroup of $G_1$ which are "too large" in a technical sense.  In the
next section we will then construct Hecke operators which avoid such 
subgroups, allowing us to prove \cref{thm:main2} in the last section.

Since the calculations will take place in the group $G_1$ rather than 
the entire group $\Gi$, \emph{for this section only} we write $\GG$ for this
group (recall that it is defined over $E$) and $G$ for the group
$\GG_1(E_w)\isom\SL_2(\C)$ that we usually denote $G_1$.
Similarly we will drop the subscript $1$ for algebraic and Lie subgroups of
$G$, especially $H \isom \SL_2(\R)$ and $M\isom U(1)$.  We write $\R$ for
$E_w$ and $\C = N_w$.

\begin{defn}\label{def:smallgroup}
Let $R<G$ be an (abstract) subgroup. We say that $R$ is \emph{small} (in $G$) 
if there exist a finite extension $K\subset \C$ of $E$ and a finite-index
subgroup $R'<R$ such that $R'$ is contained in the $K$-points
$\BB(K)$ of a $K$-subgroup $\BB$ of $\GG$ of one of the following forms:
\begin{enumerate}
\item Multiplicative type: $\BB$ is diagonable over $\C$.
\item $\SL_2$-type: $\BB$ is $\GG(K)$-conjugate to a subgroup of
$\HH\times_E K$.
\end{enumerate}
\end{defn}

\begin{rem}\label{rem:SU2ofSL2type} Observe that $\GG(\C) \isom \SL_2(\C) \times \SL_2(\C)$ in such
a way that the image of $\HH(\C)$ is the subgroup
$\{ (g,g) \mid g\in\SL_2(\C)\}$
(with $c_w$ acting on the complexified group by exchanging the two factors).
This allows us to check that the subgroup $\SU(2) = \{ g\in G=\SL_2(\C) \mid c_w(g) = {}^t g^{-1}\}$ is also of $\SL_2$-type: its complexification is
the subgroup $\{ (g,{}^t g^{-1}) \mid g\in\SL_2(\C)\}$ which is conjugate to
$\HH$ by the element
$\left(I_2,\begin{pmatrix}0 & -1 \\ 1 & 0\end{pmatrix}\right)$.
\end{rem}

We will be interested in the stabilizers of algebraic varieties.  On the one
hand they will arise for us as abstract subgroups stabilizing sets of points.
On the other hand, we will also be analyzing them as algebraic groups.  To 
connect the two points of view we rely on the following:

\begin{thm}[{\cite[Cor.\ 1.81]{Milne:AlgebraicGroups}}]
Let $\GG$ be an algebraic group acting on an algebraic variety
$\XX$.  Let $\PP<\XX$ be a closed subvariety.  Then the stabilizer
$\bSS = \{ g\in \GG \mid g\PP=\PP \}$ is a closed subgroup of $\GG$.
Furthermore if $K$ is an extension of the field of definition such that
$\PP(K)$ is Zariski-dense in $\PP$ then
$\bSS(K)=\{ g\in \GG(K) \mid g\PP(K) = \PP(K)\}$.
\end{thm}

\begin{defn}\label{def:smallmfd} We say that a submanifold $L\subset G$ has
\emph{small stabilizers} if its Zariski closure $\bLL\subset\GG$ over $\bar{\QQ}$ is defined
over a finite extension $F'/E$ contained in $\R$ and we have
\begin{itemize}
\item For every $F'$-subvariety $\PP\subset\mathbb{L}$ such that $\PP(\R)$
is Zariski-dense in $\PP$, the stabilizer
$$
S_P=\{s\in \GG(F')\mid s\PP(\R)=\PP(\R)\}, 
$$
is a small subgroup of $G$.
\end{itemize}
\end{defn}

We begin our analysis of stabilizers of subvarieties of the submanifolds $gHM$
with the observation that $HM$ is, in fact, an $E$-subvariety,
as the reader can verify by writing explicit algebraic equations in the real
and complex parts of the matrix entries.
Each translate $gHM$ is then also a real subvariety which, in the
case of interest where $g$ has algebraic entries, is defined over a number
field.

As a preliminary step we rule out one possible kind of stabilizer:

\begin{lem}\label{lem:stablem0}
Let $P\subset L=gHM$ be a real subvariety, and let $\bSS$ be its stabilizer
and $S = \bSS(\R)^\circ$ (identity component in the analytic topology).
Then $S$ is one of:
\begin{enumerate}
\item the real points of an algebraic torus; or 
\item compact (as a Lie group); or
\item contained in a conjugate of $H=\SL_2(\R)$.
\end{enumerate}
\end{lem}

\begin{proof}
We have $\dim S \leq \dim P \leq \dim L = 4$, and since $L$ itself is connected
but not the coset of a subgroup, the dimension of $S$ is actually at most $3$.

Other than possibilities 1-3 above the only other closed connected subgroups
of $\SL_2(\C)$ of dimension at most $3$ are solvable with unipotent radical
of dimension $2$, so we suppose $S^\circ$ is of that form.
Conjugating $S$ (and translating $L$ and $P$ appropriately) we may
assume that $S$ contains all the matrices of the form
$u(z)=\begin{pmatrix}1&z\\0&1\end{pmatrix}$ for $z\in\C$,
which we can now rule out by by showing that no orbit of this group
is contained in any submanifold $L$ as above, let alone in
a submanifold $P$.

Thus fix $l = ghm$ for some $g\in G$, $h\in H$, $m\in M$.  Setting $y=gh$
that $u(z)ghm\in gHM$ is equivalent to
$$
y^{-1}u(z)y\in HM\,,
$$
and for $y=\begin{pmatrix}a&b\\c&d\end{pmatrix}$ the left hand side is 
$$y^{-1}u(z)y = I_2 + \begin{pmatrix}cd&d^2\\-c^2&-cd\end{pmatrix}z\,.$$

Now the only complex subspace of the tangent space at the identity of $HM$
is the Lie alebra of $AM$, that is the diagonal matrices, whereas the image
of the derivative of the map $z\mapsto y^{-1}u(z)y$ is never of this form
since $c,d$ can't both vanish.
\end{proof}

We can now prove the following assertion:
\begin{prop}\label{prop:stablem}
Let $F'\subset \RR$ be a finite extension of $E$, and let $g\in \GG(F')$. 
Then the submanifold $L=gHM$ of $G$ has small stabilizers. 
\end{prop}

\begin{proof}
We observed earlier $L$ is Zariski-closed and defined over $F'$; let $\bLL$
be the underlying subvariety, $\PP$ an $F'$-subvariety with dense real points,
and $\bSS$ its stabilizer.  We need to show that $\bSS(F')$ is small,

If we are in the first possibility in \cref{lem:stablem0} then $\bSS$ is
diagonable, hence of multiplicative type.  In the second possibility $S$
is contained in a maximal compact subgroup, i.e. a conjugate of $\SU(2)$, so
either $S$ is conjugate to $\SU(2)$ (over a number field by
\cref{lem:algebraicconjugacy}) and hence of $\SL_2$-type as observed in 
\cref{rem:SU2ofSL2type}, or $S$ is contained in a torus and hence is of
multiplicative type.
The only remaining possibility is that $S$ is conjugate to a subgroup of
$\SL_2(\R)$, which follows from \cref{lem:algebraicconjugacy} as well. 

\end{proof}

\section{Constructing an amplifier}\label{sec:amp}
Given a Hecke eigenfunction $\phi \in L^2(X)$ and a small subgroup $S$ of $G_1$
, in section this we construct a Hecke operator which acts on $\phi$ with
large eigenvalue (making it an ``amplifier'' for $\phi$) while at the same
time avoiding the orbits of $S$ in products of trees $\GG_1(\Qp)/\GG_1(\Zp)$.

We begin with the observation that having constructed the factorization
$\tGG = \prod_i \tGG_i$ over $E$, we can now extend the field $N$ to contain
the field $K$ evincing the smallness of $S$ as in \cref{def:smallgroup} in
such a way that $E$ contains the field of definition $F'$ of $S$.
Thus for the rest of the paper the fields $E$ and $N$ and the factorizations
of $G$ will depend on $S$ (and ultimately on the subvariety of $G$ we are
trying to avoid) rather than just on $\GG$ as in the first run through
\cref{sec:notation}.  Since we are ruling out the subvarities one-by-one
this is not a problem.

Recall from that section that the group $\GG_1$ is defined over $E$
so that $\GG_1(E) = \SL_2(E_1)$ where since we assume $G_1 = \SL_2(\C)$
we have $E_1 = N$, and in that case for each prime $p$ splitting completely
in $N$ we can fix a place $w_p$ of $N$ lying over $p$ so that
$E_1\otimes_E N_{w_p} \isom \Qp\times \Qp$ or equivalently so that 
$\Gpo=\GG_1(E_{w_p})\isom \SL_2(\Qp)\times\SL_2(\Qp)$ (with $c_w$ exchanging
the two factors).  Then
$$G_p = \prod_{v\mid p} G_v = \prod_i \GG_i(E_{w_p}) = \Gpo \times \Gpgt\,,$$
and omitting finitely many primes we may also assume that with this
identification the maximal compact subgroups $K_{p,i}$ isomorphic to
$\SL_2(\Zp)$ or $\SL_2(\Zp)^2$ of the factors are factors of the adelic open
compact subgroup $\Kf$.  Let $K_p$ be their product, so under the assumptions this section, $K_p\isom \SL_2(\Zp)^2$. 

\begin{defn}A Hecke operator $\tau\in \calH_p$ \emph{one-sided} if
it is supported in the image of the first factor of $G_{p,1}$.
\end{defn}

Identifying $\Gpo$ with $\SL_2(\Qp)^2$ and $K_p$ with $\SL_2(\Zp)^2$
the \emph{basic Hecke operators} $\tau_{p^j}$ will be the one-sided Hecke
operator correspoding to the characteristic function of the double coset
$$\SL_2(\Zp)\begin{pmatrix}p^j&0\\0&p^{-j}\end{pmatrix}\SL_2(\Zp)\,,$$
that is summing over the sphere of radius $2j$ in the Bruhat-Tits tree of
the first factor of $\Gpo$.

\begin{lem}\label{lem:onesided} Let $S\subset \GG_1(E)$ be a small subgroup.
Then there is a constant $C$ (depending on $S$) such that for all primes $p>C$
(splitting completely in $N$) the $S$-orbit of the identity coset in
$G_p/K_p$:
\begin{enumerate}
\item Does not meet the support of any one-sided $\tau_{p^j}$,
      if $S$ is of $\SL_2$-type.
\item Meets the support of any one-sided $\tau_{p^j}$ at most $C$ times, if
      is of multiplcative type.
\end{enumerate}
\end{lem}
\begin{proof}
Suppose first that $S$ is of $\SL_2$-type and let
$g\in\GG(K)$ be such that $g\HH(K) g^{-1} \cap S$ is of finite index $C$
in $S$ with coset representatives $\{s_1,\ldots,s_C\}$.
Since the prime $p$ splits completely in the field $K\subset N$,
the element $h$ embeds in $\GG_1(E_{w_p})$; omitting finitely many primes
we may also assume that $g$ and the $s_i$ all lies in $\Kpo$.

Now any element of the $S$-orbit has the form $s_i g(h,h)g^{-1}$ for some
$h\in\SL_2(\Qp)$.  Assuming $s_i,g\in K_{1,p}$ if this meets the support of a
one-sided Hecke operator the second coordinate of this element must lie in the
compact subgroup $\SL_2(\Zp)$ which by the $KAK$ decomposition in $\SL_2(\Qp)$
forces $h \in \SL_2(\Zp)$ the element is in the identity coset.

Suppose now instead that $g\TT(K) g^{-1} \cap S$ is of finite index $C$ in $S$
where $\TT$ is the diagonal torus of $\GG_1$.  Again assume that $g\in K_p$;
since our torus is $K$-split and $p$ splits completely in $K$ the orbit 
of $g\TT(\Qp\times\Qp)g^{-1}$ is the product of two apartments (=geodesics)
passing through the origin of the trees $\SL_2(\Qp)/\SL_2(\Zp)$.
In particular in the first coordinate the orbit meets the translates by
$s_i^{-1}$ of each sphere (the support of $\tau_{p^j}$) at most twice.
\end{proof}

Next we construct the standard amplifier in $\SL_2(\Qp)$.  The fundamental
observation (often attributed to Iwaniec) is that it is impossible for the
eigenvalues of $\tau_p$, $\tau_{p^2}$ to be simultaneously small.

\begin{lem}\label{lem:iwaniec} Given $\phi$ we can choose $\tau$ to be either
$\tau_p$ or $\tau_{p^2}$ so that the corresponding eigenvalue $\lambda$
satisfies $\abs{\lambda} \gg \#\supp(\tau)^{1/2}$.
\end{lem}
\begin{proof}
Let $\lambda_{p^j}$ be the eigenvalue of $\tau_{p^j}$ acting on $\phi$.
A direct calculation in the tree gives the convolution identity
\begin{align}\label{eqn:heckealgebra}
\tau_p^2     &= p(p+1)+(p-1)\tau_p+\tau_{p^2}\,,\\
\tau_{p^2}^2 &= p^3(p+1) + p^2(p-1)\tau_p + p(p-1)\tau_{p^2} + (p-1)\tau_{p^3}
               +\tau_{p^4}\,.
\end{align}
The first identity implies
$(\lambda_p)^2 = p(p+1) + (p-1)\lambda_p + \lambda_{p^2}$, so
at least one of $\abs{\lambda_p}\gg p$ or $\abs{\lambda_{p^2}}\gg p^2$
must hold with the implied constants absolute.
\end{proof}

Combining the spectral calculation and the control of intersections we have
\begin{cor}[Local construction]\label{localamp}
Let $S<\GG(F)$ be a subgroup, and assume the image of $S$ under the projection 
to $\GG_1(E)$ is small.  Then there exist an absolute constant $L$,
a constant $C$ (depending on $S$), and a set of positive
density $P_S\subset \calP$ such that for every prime $p\in P_S$
there exists a finite set $J_p \subset \calH_p$ of basic Hecke operators
such that:
\begin{enumerate}
\item \label{localamp:support} For each $h_p\in J_p$,
$$p^\ell \ll \#\supp(h_p)\ll p^\ell $$
for some $\ell \ll L$ where the constants are absolute (we can take $L=4$).
\item \label{localamp:inftynorm} For $x_p \in \Gp/\Kp$ other than the origin
we have $\abs{(h_p \star h_p^*)(x_p)} \ll p^{\ell-1}$.
\item \label{localamp:intersection} For each $h_p\in J_p$,
the number of intersections of the $S$-orbit
in $G_p/K_p$ and the support of any of $h_p$,$h_p^*$,$h_p\star h_p^*$
is bounded above by $C$.
\item \label{localamp:eigenvalue} For each Hecke eigenfunction
$\phi\in L^2(X)$, at least one $h_p\in J_p$
acts on $\phi$ with eigenvalue $\lambda_p$ satisfying
$$\abs{\lambda_p} \gg \left(\#\supp h_p\right)^{1/2}\,.$$
\end{enumerate}
\end{cor}
 
\begin{proof}
\cref{lem:onesided} and \cref{lem:iwaniec} together show that the claim holds
for all primes which are large enough (depending on $S$) and split completely
in a number field $N$ which depends on $S$ (and $\GG$).
The Chebotarev Density Theorem shows that the set of split primes has
positive density.
\end{proof}
 
Our global amplifier will amplify a specified Hecke-eigenfunction $\phi$.
However, we need some control of its action on its orthogonal complement
as well.  Accordingly for a self-adjoint Hecke operator $\tau\in \calH$
we denote by $c=c(\tau)$ the smallest non-negative constant
such that the spherical transform of $\tau$ on the unitary dual is bounded
below by $-c$.  In particular the spectrum of $\tau$ acting on $L^2(X)$
is contained in $[-c,\infty)$ and therefore (a fact which can be taken as a
not-quite-equivalent definition) we have for all $R\in L^2(X)$ that
\begin{equation}\label{eqn:defect}
\left<\tau.R,R\right>\ge-c\norm{R}_2^2.
\end{equation}

\begin{prop}[Global construction]\label{globalamp} 
Continuing with the hypotheses of \cref{localamp} let also $\epsilon>0$.
Then there is $Q=Q(S,\epsilon)$ such that
for every Hecke-eigenfunction $\phi\in L^2(X)$, 
there exists $\tau\in \Span_\C \left(\cup_{p\leq Q} J_p\right)$ acting
on $\phi$ with eigenvalue $\Lambda$ satisfying
\begin{enumerate}
\item $\supp(\tau)$ meets the $S$-orbit of the identity in $\prod_p G_p/K_p$
in at most $\frac{\epsilon \Lambda}{\norm{\tau}_\infty}$ points.
\item $c(\tau) \leq \epsilon \Lambda$.
\end{enumerate}
\end{prop}

\begin{proof}
For $Q \in \Z_{\geq 1}$ and $\ell \leq L$ let $\Ps_\ell$ be the set of
primes $p\in [Q,2Q] \cap P_S$ for which there is an operator $h_p$ with
eigenvalue $\lambda_p$ as in the conclusion of \cref{localamp} with
support of size comparable to $p^\ell$, and fix $\ell$ so that $\Ps=\Ps_\ell$
consists of at least $1/L$ of the primes on in $[Q,2Q] \cap P_S$.
We then have $\# \Ps \gg Q/\log Q$ where the constant depends on $S$
through the density of $P_S$.

For each $p\in \Ps$ let $\zeta_p$ be a complex number of magnitude $1$
such that $\zeta_p \lambda_p $ is a positive real number, and finally set
$$
\tau_1=(\sum_{p\in\Ps}\zeta_p h_p)*(\sum_{p\in\Ps} \zeta_p h_p)^*
$$
and 
$$
\tau=\tau_\phi=\tau_1-\tau_1(1)\delta, 
$$
where $\delta$ is the identity element of the (full) Hecke algebra, and
$\tau_1(1)$ is the value of the function $\tau$ at the identity coset.
In other words $\tau$ is obtained from $\tau_1$ by restricting away from a
single point.

It is clear that $\tau$ is self-adjoint. To compute its $\ell^\infty$ norm
(as a function on $\Kf\backslash\GAf/\Kf$) we start with the fact that,
as functions on that space, we have the pointwise identity
$$\tau_1 \leq \sum_{p\in \Ps} h_p \star h_p +
\sum_{p<q \in \Ps} \left(h_p \star h_q^* + h_q \star h_p^*\right).$$
Since the summands on the right have disjoint supports, it suffices to
estimate each separately.  First, if $p\neq q$ then the support of the
convlution is the product of the supports since $G_p$ and $G_q$ are disjoint
commuting subgroups, so those terms are bounded by $2$.  At a single
prime $p$ the local construction gives the bound $p^{\ell-1}$ for the values
of $h_p\star h_p$ away from the origin, and we conclude that
$\norm{\tau}_\infty\ll Q^{\ell-1}$.

Next, the eigenvalue of $\tau_1$ is clearly
$$(\sum_{p\in\Ps} \abs{\lambda_p})^2 \gg \left(\sum_{p\in \Ps} \sqrt{\#\supp(h_p)}\right)^2 \gg \left(\#\Ps Q^{\ell/2}\right)^2 \gg \frac{Q^{2+\ell}}{\log^2 Q},$$
and since $h_ph_q^*(1)=0$ whenever $p\ne q$ and $h_p^2(1)=\#\supp(h_p)$ we have
$$\tau_1(1) = \sum_{p\in\Ps} \#\supp(h_p)\ll \frac{Q^{1+\ell/2}}{\log Q}\,.$$
Subtracting the two gives
$$\Lambda \gg \frac{Q^{2+\ell}}{\log^2 Q}\,.$$

For primes $p,q$ let $T_{p,q}$ be the set of points in
$\supp h_p \star h_q^*$ that meet the $S$-orbit of the identity.
When $p=q$ these sets are of uniformly bounded size by
the third item of \cref{localamp}, whereas when $p\neq q$ this is the product
of the corresponding subsets of $G_p/K_p$ and $G_q/K_q$ so again uniformly
bounded.  We conclude that the product of $\norm{\tau_\infty}$ with
the total number of intersections satisfies
\begin{equation}\label{eqn:count-intersections} 
\frac{\norm{\tau_\infty}\sum_{p,q\in\Ps} \# T_{p,q}}{\Lambda}
\ll \norm{\tau_\infty}\frac{\#\Ps^2}{\Lambda}
\ll Q^{-1}\,.
\end{equation}

Similarly since $\tau_1$ is positive definite, $c(\tau) = \tau_1(1)$ and hence
\begin{equation}\label{eqn:count-positivity}
\frac{c(\tau)}{\Lambda} = \frac{\tau_1(1)}{\lambda} \ll Q^{-1-\ell/2}\log Q\,.
\end{equation}

Finally if we take $Q$ large enough we can can ensure both right-hand-sides
of \cref{eqn:count-intersections,eqn:count-positivity} 
are less than $\epsilon$.
\end{proof}

\section{Non-concentration on homogenous submanifolds}\label{sec:proofs}
\newcommand\piG{\pi_G}
\newcommand\piA{\pi_{\Adele}}

We have shown that to rule out non-Haar components it suffices to rule out
components supported in images in in $X=\Gamma\backslash G$ of submanifolds
$gHM \subset G$, which we finally do in this section.
As in \cref{sec:homogeneous} instead of bounding $\mu(\Gamma gHM)$ we will bound
$\mu(\Gamma U_\epsilon)$ where $U\subset gHM$ is a bounded neighborhood.
Unlike the positive entropy arguments we now need to treat $U$ as a subset of
$gHM$ rather than try for additional uniformity by fixing a single
$U\subset HM$ and translating by $g$ later.

Since will calculate in $G$ and $\GA$ but ultimately make statements
about subsets of $X$ we introduce the notation
$\piG\colon G\to X =\Gamma\backslash G$ and
$\piA\colon\GA \to X = \GF\backslash\GA/\Kf$ for the quotient maps.

\begin{defn}An \emph{algebraic piece} of $G$ will be a triple
$(U,\bLL,F')$ where $F'/E$ is a finite extension contained in $E_w$,
$\bLL\subset \tGG_1$ is an irreducible subvariety defined over $F'$ and
irreducible over $E_w$, and $U\subset (\bLL\times\tGgt)(E_w)$ is Zariski-dense,
and also bounded and relatively open in the \emph{analytic} topology.
\end{defn}

\newcommand{\tLL}{\tilde{\bLL}}

This parametrization is redundant ($U$ determines $\bLL$) but it is easier to
directly keep track of $\bLL$ and its field of definition.  We also
write $\tLL = \bLL\times\tGgt$

Consider now a translate of an algebraic piece $U$ by some element
$\gf\in\GAf$.  Let $\gamma\in\GF$ be such that $\gamma\gf \in \Kf$. Then
$$\piA(U \gf) = \piA(\gamma U\gf) = \piG(\gamma U)$$

For a subset $U\subset G$ write $\bar{U}$ for its closure in the analytic
topology and $\bar{U}^1$ for its projection to $G_1$.

\begin{defn} We say the translate $U \gf$ is \emph{transverse} to $U$
if for each $\gamma\in\GF$ with $\gamma\gf\in\Kf$ the $F'$-Zariski closure
of the intersection $\gamma \bar{U}^1 \cap \bar{U}^1\subset G_1$ is of smaller
dimension than $\bLL$.
Abusing notation we say that $\gf$ itself is transverse to $U$,
and otherwise say that it is \emph{parallel} to $U$.
\end{defn}
Observe that the definition only depends on the coset of $\gf$ in $\GAf/\Kf$.

\begin{lem}\label{parallel-stab} Let $\gamma\in\GF$ and $\gf\in\GAf$ be such
that $\gamma\gf\in\Kf$.  Suppose further that the $F'$-Zariski closure of
$\gamma \bar{U}^1 \cap \bar{U}^1$ is of dimension $\dim\bLL$.
Then $\gamma \bLL=\bLL$.\qed
\end{lem}
For a function $\phi\in L^2(X)$ and a measurable subset $N\subset G$
(\resp $N\subset \GAF$) we write $\phi_N$ for the restriction
of $\phi$ to $\piG(N)$ (\resp to $\piA(N)$).

\begin{prop}\label{inductionstep}
Let $(U,\bLL,F')$ be an algebraic piece of $G$ and let $\tau$
be a self-adjoint Hecke operator.  Write $P$ for the set of parallel elements
in the support of $\tau$.
Then there exists a finite collection $\calV$ of algebraic pieces
$(V,\bLL',F'')$ of $G$ such that the $\bLL'$ are irreducible proper
$F''$-subvarieties of $\bLL$ and such that for every $\delta>0$ there exists
$\epsilon>0$ such that for any
eigenfunction $\phi\in L^2(X)$ of $\tau$ with eigenvalue $\Lambda>0$, we have 
$$
\norm{\phi_\Ue}^2\le
\frac{1}{\Lambda}\left(\#P\norm{\tau}_\infty+c(\tau) + 
\frac{\norm{\tau}_\infty}{\norm{\phi_\Ue}^2}
\sum_{V\in\calV}\norm{\phi_\Vd}^2\right)\,.
$$
Here $c(\tau)$ is the constant defined in \eqref{eqn:defect}.
\end{prop}
For the proof we will use the following easy and elementary observation:
\begin{lem}\label{easylem}
Suppose $X$ is a metric space and $A,B\subset X$ bounded subsets of $X$. Then for each $\delta>0$ there exists $\epsilon>0$ such that 
$
A_\epsilon\cap B_\epsilon\subset (\bar{A}\cap \bar{B})_\delta
$.
\qed
\end{lem}

\begin{proof}[Proof of \cref{inductionstep}]
We estimate the expression
\begin{equation}\label{pretrace}
\left\langle\tau.\phi_\Ue,\phi_\Ue\right\rangle
\end{equation}
in two different ways: a geometric upper bound and a spectral lower bound.

On the geometric side let $S=P\sqcup T\subset\GAf$ be a set of
representatives for the support of $\tau$, partitioned into the subsets of
transverse and parallel elements.  By the triangle inequality 
$$
\left\langle\tau.\phi_\Ue,\phi_\Ue\right\rangle\le \norm{\tau}_\infty
\sum_{\gf\in S} \int_X\abs{\phi_\Ue(g\gf)\overline{\phi_\Ue(g)}}dg.
$$
When $\gf\in P$ we naively apply Cauchy--Schwartz and the unitarity of
the right $\GAf$-action to get 
\begin{equation}\label{cauchyschwartz}
\int_X\abs{\phi_\Ue(g\gf)\overline{\phi_\Ue(g)}}dg \leq \norm{\phi_\Ue}^2
\end{equation}

Now suppose $\gf\in T$.  The images of $\Ue$ and $\Ue \gf$ in $X$ intersect if
and only if there is $\gamma\in\GF$ such that $\gamma\gf\in\Kf$ and
such that $\gamma \Ue \cap \Ue \neq \emptyset$.
These $\gamma$ belong to the intersection
$\GF \cap (\Ue\Ue^{-1})\times (\Kf \gf^{-1})$ of a discrete subset and 
a compact subset of $\GA$ so there are finitely many of them. 

By \cref{easylem} for each such $\gamma$ and any $\delta>0$ the boundedness
of $U$ ensures the existence of $\epsilon>0$ (depending also on $\gamma$)
such that
$$
\Ue\cap \gamma \Ue\subset (\bar{U}\cap \gamma \bar{U})_\delta\,.
$$
Since we assumed that $\gf$ is transverse the $F'$-Zariski closure $\MM$ of
$\bar{U}^1\cap \gamma \bar{U}^1$ in $\tGG_1$ has $\dim \MM < \dim \bLL$.
The irreducible components of $\MM\times_{F'} E_w$ are defined over some finite
extension $F''$ of $F'$, and we can then cover $\bar{U}\cap \gamma \bar{U}$
with finitely many algebraic pieces $(V,\bLL',F'')$ where $\bLL'$ is an
irreducible component of $\MM\times_{F'} F''$.  The union of the
$\delta$-neighborhoods
of these pieces then covers the intersection $\Ue \cap \gamma \Ue$.

Now for one piece $V$ we have
\begin{align*}
\int_\Vd\abs{\phi(g\gf)\overline{\phi(g)}}dg
  &\le \frac12\left( \int_\Vd\abs{\phi(\gamma g)}^2dg
                    +\int_\Vd\abs{\phi(g       )}^2dg\right) \\
  &=   \frac12\left( \int_{\gamma^{-1}\Vd}\abs{\phi(g)}^2 dg
                    +\int_{           \Vd}\abs{\phi(g)}^2 dg\right). 
\end{align*}
Accordingly, let $\calV$ denote the set of pieces $V$ and $\gamma^{-1}V$
arising as components of $\MM$ as $\gf$ ranges over $T$ and $\gamma$ ranges
over the finite set causing intersections
(note that $\gamma^{-1}\Vd = (\gamma^{-1}V)_\delta$).

Since this set of pieces is finite we can choose $\epsilon$ small enough
for all of them.  Combining the bounds for parallel and transverse
intersections then gives the geometric-side estimate
\begin{equation}\label{geomside}
\left\langle\tau.\phi_\Ue,\phi_\Ue\right\rangle \leq
\#P\norm{\tau}_\infty \norm{\phi_\Ue}^2 +
\norm{\tau}_\infty \sum_{V\in \calV} \norm{\phi_\Vd}^2\,.
\end{equation}

We next obtain a spectral lower bound.  Returning to \cref{pretrace},
to the extent $\phi_\Ue$ "approximates" $\phi$
we expect $\tau \phi_\Ue$ to be approximately $\Lambda \phi_\Ue$.  This is
not literally true; instead we write $\phi_\Ue = a \phi + \phi^\perp$
for some $\phi^\perp$ orthogonal to $\phi$. Here
$a = \left\langle \phi_\Ue,\phi\right\rangle = \norm{\phi_\Ue}^2$ since
$\phi$ is normalized.  Since also $\norm{\phi^\perp} \leq \norm{\phi_\Ue}$
we have
\begin{align}\label{spectralside}
\left\langle\tau.\phi_\Ue,\phi_\Ue\right\rangle
 &= a^2 \left\langle\tau.\phi,\phi\right\rangle
   +2\Re\left\langle\tau.\phi,\phi^\perp\right\rangle
   +    \left\langle\tau.\phi^\perp,\phi^\perp\right\rangle \nonumber\\
 &= a^2 \Lambda + 2\Lambda \Re\left\langle \phi,\phi^\perp\right\rangle
   +    \left\langle\tau.\phi^\perp,\phi^\perp\right\rangle \\
 &\geq \norm{\phi_\Ue}^4 \Lambda - c(\tau) \norm{\phi_\Ue}^2\,.\nonumber
\end{align}
   
Putting \cref{geomside,spectralside} together and dividing by
$\Lambda\norm{\phi_\Ue}^2$ finally established the Proposition.
\end{proof}

We can now prove out main result.
\begin{thm}\label{inductionargument}
Let $\mu$ be a weak-$*$ limit of normalized Hecke-eigenfunctions on $X$
and let $(U,\bLL,F')$ be an algebraic piece of $G$ so that $\bLL$ has
small stabilizers.  Then $\mu(\piG(U)) = 0$.
\end{thm}
\begin{proof}
We show by induction on $\dim\bLL$ that for every $\eta>0$ there is
$\epsilon>0$ such that $\norm{\phi_\Ue})^2 \leq \eta$ for all Hecke
eigenfunctions $\phi$.  It would follow that $\mu(\piG(U)) \leq \eta$
for all limits.

When $U$ and $\bLL$ are empty ("dimension $-1$") there is nothing to prove,
so we assume $\dim\bLL\geq 0$.
Let $S<\GF$ be the stabilizer of $\bLL$ which is small by hypothesis.
With $\alpha>0$ to be chosen later let $\tau$ be the Hecke operator
constructed in \cref{globalamp} so that $\tau\phi=\Lambda\phi$ and
$\#P\norm{\tau}_\infty,c(\tau) \leq \alpha \Lambda$ where $P$ is the set of
elements in the support of $\tau$ which lie on the $S$-orbit.
By \cref{parallel-stab} every element in the support of $\tau$
parallel to $U$ lies in the $S$-orbit so \cref{inductionstep} produces
a finite set $\calV$ of algebraic pieces contained in $U$ (hence themselves
having small stabilizers) and such that for any $\delta>0$
there is $\epsilon>0$ such that
$$
\norm{\phi_\Ue}^2\le
\frac{1}{\Lambda}\left(\#P\norm{\tau}_\infty+c(\tau) + 
\frac{\norm{\tau}_\infty}{\norm{\phi_\Ue}^2}
\sum_{V\in\calV}\norm{\phi_\Vd}^2\right)\,.
$$

By induction we can choose $\delta$ small enough so that for each $V\in\calV$
we have
$\norm{\phi_\Vd}^2 \leq \frac{1}{\#\calV}\alpha$.
With this choice (and the corresponding $\epsilon$) we have
$$\norm{\phi_\Ue}^2 \leq \alpha\left(2 + \frac1{\norm{\phi_\Ue}^2}\right)\,,$$
and hence $\norm{\phi_{U_{\delta'}}}_2^2\leq 3\sqrt{\alpha}$
(if we always choose $\alpha<1$).
The Theorem follows upon choosing $\alpha$ small enough.
\end{proof}

\begin{proof}[Proof of \cref{thm:main2}]
Without loss of generality $i=1$.
Let $g \in G = \prod_i G_i$ have its $1$st coordinate in $\tGG_1(\Qbar)$.
Then \cref{prop:stablem} shows that $L=g(H_1 M_1)\Ggt$ has small stabilizers.
By \cref{inductionargument} for every bounded open $U\subset L$ we have
$\mu(\piG(U))=0$ and covering $L$ by countably many such $U$ we conclude that
$\mu(\piG(L))=0$.
\end{proof}

\bibliographystyle{plain}
\bibliography{que,algebra,aut_forms}
\end{document}